\documentclass[11pt]{article}

\usepackage{amsmath,amsthm,amssymb}

\usepackage{graphicx}
\graphicspath{{pictures/}} % path to the folder with the graphic files

\usepackage{geometry}
\newtheorem{thm}{Theorem}

\numberwithin{equation}{section}

\usepackage{booktabs} % Nice tables

\title{An optimal class of eighth-order iterative methods based on Kung and Traub's method with its dynamics}

\author{Gunar Matthies $^a$\thanks{gunar.matthies@tu-dresden.de} 
\and 
Mehdi Salimi $^{b}$\thanks{mehdi.salimi@tu-dresden.de} 
\and 
Somayeh Sharifi $^{c}$\thanks{s.sharifi@iauh.ac.ir}
\and
Juan Luis Varona $^{d}$\thanks{jvarona@unirioja.es}\\[9pt]
\small
$^{a}$Institut f{\"u}r Numerische Mathematik, Technische Universit{\"a}t Dresden, Germany\\[-3pt]
\small
$^{b}$Center for Dynamics, Department of Mathematics, Technische Universit{\"a}t Dresden, Germany\\[-3pt]
\small
$^{c}$Young Researchers and Elite Club, Hamedan Branch, Islamic Azad University, Hamedan, Iran\\[-3pt]
\small
$^{d}$Departamento de Matem{\'a}ticas y Computaci{\'o}n, Universidad de La Rioja, Logro{\~n}o, Spain
}

\date{}
\begin{document}
%--------------------------------------------------------------%

%--------------------------------------------------------------%
\maketitle
%--------------------------------------------------------------%

\begin{abstract}
In this paper, we present a three-point without memory
iterative method based on Kung and Traub's method for solving
non-linear equations in one variable. The proposed method has
eighth-order convergence and costs only four function evaluations
each iteration which supports the Kung-Traub conjecture on the
optimal order of convergence. Consequently, this method possesses
very high computational efficiency. We present the construction,
the convergence analysis, and the numerical implementation of the
method. Furthermore, comparisons with some other existing optimal
eighth-order methods concerning accuracy and basins of attraction
for several test problems will be given.
\medskip

\medskip

\end{abstract}

%--------------------------------------------------------------%
\section{Introduction}
%--------------------------------------------------------------%

Solving nonlinear equations is a basic and extremely valuable tool
in all fields in science and engineering. One can distinguish
between two general approaches for solving nonlinear equations
numerically, namely, one-step and multi-step methods. Multi-step
methods overcome some computational issues encountered with
one-step iterative methods. Typically they allow us to achieve a
greater accuracy with the same number of function evaluations.
Important aspects related to these methods are order of
convergence and optimality. Therefore, it is favorable to attain
with fixed number of function evaluations each iteration step
a convergence order which is as high as possible. A central role
in this context plays the unproved conjecture by Kung and Traub~\cite{Kung}.
It states that an optimal multi-step method without memory which uses
$k+1$ evaluations could achieve a convergence order of~$2^{k}$.
Considering this conjecture, many optimal two-step and three-step
methods have been presented.

In the recent years, a large number of multi-step methods for finding
simple roots $x^*$ of a nonlinear equation $f(x)=0$ with a scalar function
$f: D\subset\mathbb{R} \to \mathbb{R}$ which is defined on
an open interval~$D$ (or $f: D\subset\mathbb{C} \to \mathbb{C}$
defined on a region $D$ in the complex plane~$\mathbb{C}$) have been developed
and analyzed for improving the convergence order of classical methods.
The Newton-Raphson iteration $x_{n+1} := x_n - \frac{f(x_n)}{f'(x_n)}$
is probably the most widely used algorithm for finding roots.
It is of second order and requires two evaluations for each iteration
step, one evaluation of $f$ and one of~$f'$.
The Newton-Raphson iteration is an example of a one-point
iteration, i.e., in each iteration step the evaluations are taken
at a single point. The basic optimality theorem for one-point
iterations (see Traub~\cite[\S\,5.4]{Traub} or an improved
proof in~\cite{Kung}) shows that an analytic one-point iteration based
on $k$ evaluations is of order at most~$k$. Thus, the Newton-Raphson
iteration is an optimal one-point method with~$k=2$.

Some well known two-point methods without memory are described
e.g.\ in Jarratt~\cite{Jarratt}, King~\cite{King}, and
Ostrowski~\cite{Ostrowski}. Using inverse interpolation, Kung and
Traub \cite{Kung} constructed two general optimal classes without
memory. Since then, there have been many attempts to construct
optimal multi-point methods, utilizing e.g.\ weight functions, see
in particular~\cite{Babajee, Bi, Chun1, Lotfi2, Petkovic,
Sharifi1, Sharifi2, Sharma1, Wang}. Here, we will construct a
class of eighth-order methods free from second order derivatives
with efficiency index $\sqrt[4]{8} \simeq 1.68179$; recall that
the efficiency index of an iterative method of order $p$ requiring
$k$ function evaluations per iteration step is defined by $E(k,p)
= \sqrt[k]{p}$, see~\cite{Ostrowski}.

The paper is organized as follows: Section~\ref{sec:description}
is devoted to introduce the ideas for the construction of the new
optimal class of eighth-order methods based on Kung and Traub's method
by using a Newton-step and suitable weight functions.
In Section~\ref{sec:convergence} we give the details of
the new methods and investigate the convergence order;
this allows to present a class of optimal three-point
methods by using suitable weight functions.
Particularizing the weight functions we construct a three-parametric
family of eighth-order optimal iterative root-finding methods.
By assigning particular values to these parameters we propose
two examples for this kind of methods.
Numerical performance and comparisons with
other methods are illustrated in Section~\ref{sec:examples}.
In Section~\ref{sec:dynamic} we approximate and visualize the basins
of attraction of the proposed method and compare them with several
existing methods, both graphically and by mean of some numerical
measures. Finally, a conclusion is provided in Section~\ref{sec:conclusion}.

%--------------------------------------------------------------%
\section{Description of the method}
\label{sec:description}
%--------------------------------------------------------------%

In this section we construct a new optimal three-point class of
iterative methods for solving nonlinear equations based on Kung
and Traub's method~\cite{Kung}.

The Kung and Traub's method is given by
\begin{equation}
\label{a1}
%\begin{cases}
%y_{n} = x_{n} - \dfrac{f(x_{n})}{f'(x_{n})}, \\[1.8ex]
%x_{n+1} = y_{n} - \dfrac{f(x_n)f(y_n)}{(f(x_n)-f(y_n))^2} \cdot \dfrac{f(x_n)}{f'(x_n)},
%\end{cases}
\left\{
\begin{aligned}
y_{n} & := x_{n} - \dfrac{f(x_{n})}{f'(x_{n})}, \\[1.8ex]
x_{n+1} & := y_{n} - \dfrac{f(x_n)f(y_n)}{(f(x_n)-f(y_n))^2} \cdot
\dfrac{f(x_n)}{f'(x_n)}.
\end{aligned}
\right.
\end{equation}
In~\eqref{a1} and all forthcoming methods, the iteration rule is used for
$n = 0, 1, \dots$ and $x_0$ denotes an initial approximation of the simple
root~$x^{*}$. The convergence order
of \eqref{a1} is four with three function evaluations in each
iteration step. Hence, this method is optimal. We intend to increase the
order of convergence and extend \eqref{a1} by mean of an
additional Newton step
\begin{equation}
\label{a2}
%\begin{cases}
%y_{n} = x_{n} - \dfrac{f(x_{n})}{f'(x_{n})}, \\[1.8ex]
%z_{n} = y_{n} - \dfrac{f(x_n)f(y_n)}{(f(x_n)-f(y_n))^2}\cdot\dfrac{f(x_n)}{f'(x_n)}, \\[1.8ex]
%x_{n+1} = z_{n} - \dfrac{f(z_n)}{f'(z_n)}.
%\end{cases}
\left\{
\begin{aligned}
y_{n} & := x_{n} - \dfrac{f(x_{n})}{f'(x_{n})}, \\[1.8ex]
z_{n} & := y_{n} - \dfrac{f(x_n)f(y_n)}{(f(x_n)-f(y_n))^2}\cdot\dfrac{f(x_n)}{f'(x_n)}, \\[1.8ex]
x_{n+1} & := z_{n} - \dfrac{f(z_n)}{f'(z_n)}.
\end{aligned}
\right.
\end{equation}
Method \eqref{a2} uses five function evaluations with convergence order
eight. Consequently, this method is not optimal. In order
to decrease the number of function evaluations, we are going to approximate
$f'(z_n)$ by an expression based on $f(x_n)$, $f(y_n)$, $f(z_n)$,
and~$f'(x_n)$, namely
\begin{equation*}
f'(z_n) \approx \frac{f'(x_n)}{J(t_n,u_n)G(s_n)},
\end{equation*}
with
$t_n := \frac{f(y_n)}{f(x_n)}$, $u_n := \frac{f(z_n)}{f(x_n)}$,
$s_n := \frac{f(z_n)}{f(y_n)}$, and suitable functions $J$ and~$G$.

Therefore, we have
\begin{equation}
\label{a3}
%\begin{cases}
%y_{n} = x_{n}-\dfrac{f(x_{n})}{f'(x_{n})}, \\[1.8ex]
%z_{n} = y_n-\dfrac{f(x_n)f(y_n)}{(f(x_n)-f(y_n))^2}\cdot\dfrac{f(x_n)}{f'(x_n)}, \\[1.8ex]
%x_{n+1} = z_n-\dfrac{f(z_n)}{f'(x_n)} J(t_n,u_n) G(s_n),
%\end{cases}
\left\{
\begin{aligned}
y_{n} & := x_{n}-\dfrac{f(x_{n})}{f'(x_{n})}, \\[1.8ex]
z_{n} & := y_n-\dfrac{f(x_n)f(y_n)}{(f(x_n)-f(y_n))^2}\cdot\dfrac{f(x_n)}{f'(x_n)}, \\[1.8ex]
x_{n+1} & := z_n-\dfrac{f(z_n)}{f'(x_n)} J(t_n,u_n) G(s_n),
\end{aligned}
\right.
\end{equation}
as iteration rule.

%--------------------------------------------------------------%
\section{Convergence analysis}
\label{sec:convergence}
%--------------------------------------------------------------%

In the following theorem, we analyze the convergence order of
method~\eqref{a3}. In particular, we find the requisites to the
weight functions $J$ and $G$ in~\eqref{a3} guarantee the requested order
eight. Although we enunciate the method for real
functions and a real root, the same can be written (with an
identical proof) if we have a complex function $f:D\subset
\mathbb{C} \to \mathbb{C}$ with a complex root~$x^{*} \in D$.

\begin{thm}
\label{thm:main} 
Let $f:D\subset \mathbb{R} \to \mathbb{R}$ be an
eight times continuously differentiable function with a simple
zero $x^{*}\in D$, and let $J:\mathbb{R}^2 \to \mathbb{R}$ and
$G:\mathbb{R \to \mathbb{R}}$ sufficiently differentiable
functions in a neighborhood of the origin. If the initial point
$x_{0}$ is sufficiently close to $x^{*}$. Then, the method defined
by \eqref{a3} converges to $x^{*}$ with order eight if the conditions
\begin{equation*}
J_{0,0}=1, \quad J_{1,0}=2, \quad J_{2,0}=8, \quad J_{0,1}=2,
\quad J_{3,0}=36,
\end{equation*}
and
\begin{equation*}
G_0=1, \quad G_1=1,
\end{equation*}
with $J_{i,j} = \frac{\partial^{i+j} J(t,u)}{\partial
t^i\partial u^j}|_{(t,u)=(0,0)}$
and $G_i = \frac{d^iG(s)}{ds^i}|_{s=0}$ are fulfilled.
\end{thm}

\begin{proof}
Let $e_{n} := x_{n}-x^{*}$, $e_{n,y}:=y_{n}-x^{*}$,
$e_{n,z} := z_{n}-x^{*}$ and
$c_{n} := \frac{f^{(n)}(x^{*})}{n!f^{'}(x^{*})}$ for $n\in \mathbb{N}$.
Using the fact that $f(x^{*})=0$, the Taylor expansion of $f$ at
$x^{*}$ yields
\begin{equation}
\label{a10}
f(x_{n}) = f'(x^*)\left(e_{n}
+ c_{2}e_{n}^{2}+c_{3}e_{n}^{3} + \cdots
+ c_{8}e_{n}^{8}\right)+O(e_n^{9})
\end{equation}
and
\begin{equation}
\label{a11}
f'(x_{n}) = f'(x^*) \left(1 +
2c_{2}e_{n}+3c_{3}e_{n}^{2}+4c_{4}e_{n}^{3} + \cdots
+ 9c_9e_n^8\right) + O(e_n^{9}).
\end{equation}
Therefore, we have
\begin{equation*}
\begin{split}
\frac{f(x_{n})}{f'(x_{n})}
&= e_{n}-c_{2}e_{n}^{2} + \left(2c_{2}^{2}-2c_{3}\right) e_{n}^{3}
+ (-4c_2^3+7c_2c_3-3c_4)e_n^4\\
& \quad + (8c_2^4-20c_2^2c_3+6c_3^2+10c_2c_4-4c_5)e_n^5\\
& \quad + (-16c_2^5+52c_2^3c_3-28c_2^2c_4
+ 17c_3c_4-c_2(33c_3^2-13c_5))e_n^6 + O(e_n^{7}),
\end{split}
\end{equation*}
and
\begin{equation*}
\begin{split}
e_{n,y} = y_n-x^* & = c_{2}e_{n}^{2}+(-2c_2^2+2c_3)e_n^3
+ (4c_2^3-7c_2c_3+3c_4)e_n^4\\
& \quad + (-8c_2^4+20c_2^2c_3-6c_3^2-10c_2c_4+4c_5)e_n^5\\
& \quad + (16c_2^5-52c_2^3c_3+28c_2^2c_4-17c_3c_4
+ c_2(33c_3^2-13c_5))e_n^6 + O(e_n^{7}).
\end{split}
\end{equation*}
We have for $f(y_n)$ also
\begin{equation}
\label{a12a}
f(y_{n}) = f'(x^*) \left(e_{n,y} + c_{2}e_{n,y}^{2}+c_{3}e_{n,y}^{3}
+ \cdots + c_{8}e_{n,y}^{8}\right) + O(e_{n,y}^{9}).
\end{equation}
Therefore, by substituting \eqref{a10}, \eqref{a11}, and
\eqref{a12a} into \eqref{a2}, we get
\begin{equation*}
\begin{split}
e_{n,z} = z_n-x^* & = (2c_2^3-c_2c_3)e_{n}^{4}-2
(5c_2^4-7c_2^2c_3+c_3^2+c_2c_4)e_n^5\\
& \quad + \left(31c_2^5-726c_2^3c_3+21c_2^2c_4-7c_3c_4
+ c_2(30c_3^2-3c_5)\right) e_n^6+O(e_n^7).
\end{split}
\end{equation*}
We get for $f(z_n)$ also
\begin{equation}
\label{a12b}
f(z_{n}) = f'(x^*)\left(e_{n,z} +
c_{2}e_{n,z}^{2}+c_{3}e_{n,z}^{3}+\cdots+c_{8}e_{n,z}^{8}\right) + O(e_{n,z}^{9}).
\end{equation}
From \eqref{a10} and \eqref{a12a}, we have
\begin{equation}
\label{a13a}
\begin{split}
t_n = \frac{f(y_n)}{f(x_n)}
& = c_2e_n+(-3c_2^2+2c_3)e_n^2+(8c_2^3-10c_2c_3+3c_4)e_n^3\\
& \quad + (-20c_2^4+37c_2^2c_3-8c_3^2-14c_2c_4+4c_5)e_n^4\\
& \quad + \left(48c_2^5-118c_2^3c_3+51c_2^2c_4-22c_3c_4
+ c_2(55c_3^2-18c_5)\right) e_n^5+O(e_n^6),
\end{split}
\end{equation}
and from \eqref{a10} and \eqref{a12b}, we obtain
\begin{equation}
\label{a13b}
\begin{split}
u_n &= \frac{f(z_n)}{f(x_n)}
= (2c_2^3-c_2c_3)e_n^3+(-12c_2^4+15c_2^2c_3-2c_3^2-2c_2c_4)e_n^4\\
& \quad + (43c_2^5-89c_2^3c_3+23c_2^2c_4-7c_3c_4
+ c_2(33c_3^2-3c_5))e_n^5+O(e_n^6).
\end{split}
\end{equation}
We get from \eqref{a12a} and \eqref{a12b}
\begin{equation}
\label{a13c}
\begin{split}
s_n = \frac{f(z_n)}{f(y_n)} & = (2c_2^2-c_3)e_n^2-2(3c_2^3-4c_2c_3+c_4)e_n^3\\
& \quad + (9c_2^4-25c_2^2c_3+7c_3^2+11c_2c_4-3c_5)e_n^4\\
& \quad + 2\left(c_2^5-18c_2^3c_3+15c_2^2c_4-9c_3c_4
+ c_2(16c_3^2-7c_5)\right)e_n^5+O(e_n^6).
\end{split}
\end{equation}
Expanding $J$ at $(0,0)$ and $G$ at $0$ yields
\begin{align}
\label{a14a}
J(t_n,u_n) & = J_{0,0}+u_nJ_{0,1}+t_nJ_{1,0} + \frac{1}{2}t_n^2J_{2,0}
+ \frac{1}{6}t_n^3J_{3,0}+O(t_n^4,u_n^2),\\
\label{a14b}
G(s_n) & = G_0+s_nG_1+O(s_n^{2}).
\end{align}
Substituting \eqref{a10}--\eqref{a14b} into \eqref{a3}, we obtain
\begin{equation*}
e_{n+1} = x_{n+1}-x^* = R_4e_n^4+R_5e_n^5+R_6e_n^6+R_7e_n^7+R_8e_n^8+O(e_n^9),
\end{equation*}
where
\begin{equation*}
\begin{split}
R_4 &= -(2c_2^3-c_2c_3)(-1+G_0J_{0,0}),\\
R_5 &= -c_2^2(2c_2^2-c_3)(-2+J_{1,0}),\\
R_6 &= -\frac{1}{2}c_2(2c_2^2-c_3) \left(-2c_3(-1+G_1)+c_2^2(-12+4G_1+J_{2,0})\right),\\
R_7 &= -\frac{1}{6}c_2^2(2c_2^2-c_3) \left(-6c_3(-2+J_{0,1})+c_2^2(-60+12J_{0,1}+J_{3,0})\right).
\end{split}
\end{equation*}
By setting $R_4=\cdots=R_7=0$ and $R_8\neq0$, the convergence
order becomes eight. Obviously, we have
\begin{equation*}
\begin{split}
J_{0,0}=1, \quad G_0=1 \quad \Rightarrow \quad R_4&=0,\\
J_{1,0}=2, \quad  \Rightarrow \quad R_5&=0,\\
J_{2,0}=8, \quad G_1=1 \quad \Rightarrow \quad R_6&=0,\\
J_{0,1}=2, \quad J_{3,0}=36 \quad \Rightarrow \quad R_7&=0.
\end{split}
\end{equation*}
Consequently, the error equation becomes in this case
\begin{equation*}
e_{n+1} = \left(c_2\left(2c_2^2-c_3\right)\left(23c_2^4-12c_2^2c_3+c_3^2+c_2c_4\right)\right) e_n^8+O(e_n^9)
\end{equation*}
which finishes the proof of the theorem.
\end{proof}

In what follows, we give some concrete explicit representations of
\eqref{a3} by choosing different weight functions satisfying the
required conditions for the weight functions $J(t_n, u_n)$ and
$G(s_n)$ of Theorem~\ref{thm:main}.

We can choose the weight functions $J(t_n, u_n)$ and $G(s_n)$ as
\begin{align}
\label{a36}
J(t_n, u_n) & = \frac{1+a t_n+(2+b) u_n+(2a+1)t_n^2+4at_n^3}{1+(a-2)t_n+b u_n+ t_n^2}
\intertext{and}
\label{a37}
G(s_n) & = \frac{1+c s_n}{1+(c-1)s_n}
\end{align}
with arbitrary $a, b, c \in \mathbb{C}$.
It is a simple task to check that the functions $J(t_n,u_n)$ and $G(s_n)$
in \eqref{a36} and \eqref{a37} satisfy the assumptions of
Theorem~\ref{thm:main} for all choices of $a,b,c$.
Hence, three-parametric family of optimal eighth-order iterative
root-finding methods is obtained.

By fixing the particular parameters $a, b, c$, we are going to give
two examples of this family of methods.

\paragraph{Method 1:}
Set $a=b=c=\dfrac{1}{2}$. Then, we get
\begin{equation}
\label{n1}
%\begin{cases}
%y_{n} = x_{n} - \dfrac{f(x_{n})}{f'(x_{n})}, \\[1.8ex]
%z_{n} = y_{n} - \dfrac{f(x_n)f(y_n)}{(f(x_n)-f(y_n))^2} \dfrac{f(x_n)}{f'(x_n)}, \\[1.8ex]
%x_{n+1} = z_{n} - \dfrac{f(z_n)}{f'(x_n)}
%\left(\dfrac{2+t_n+5u_n+4t_n^2+4t_n^3}{2-3t_n+u_n+2t_n^2} \cdot
%\dfrac{2+s_n}{2-s_n}\right)
%\end{cases}
\left\{
\begin{aligned}
y_{n} & := x_{n} - \dfrac{f(x_{n})}{f'(x_{n})}, \\[1.8ex]
z_{n} & := y_{n} - \dfrac{f(x_n)f(y_n)}{(f(x_n)-f(y_n))^2} \dfrac{f(x_n)}{f'(x_n)}, \\[1.8ex]
x_{n+1} & := z_{n} - \dfrac{f(z_n)}{f'(x_n)}
\left(\dfrac{2+t_n+5u_n+4t_n^2+4t_n^3}{2-3t_n+u_n+2t_n^2} \cdot
\dfrac{2+s_n}{2-s_n}\right)
\end{aligned}
\right.
\end{equation}
with $t_n = \frac{f(y_n)}{f(x_n)}$, $u_n = \frac{f(z_n)}{f(x_n)}$,
$s_n = \frac{f(z_n)}{f(y_n)}$.

\paragraph{Method 2:}
Set $a = \dfrac{i+1}{2}$, $b = 1+i$, and $c =
\dfrac{i-1}{2}$. So, we have
\begin{equation}
\label{n2}
%\begin{cases}
%y_{n} = x_{n} - \dfrac{f(x_{n})}{f'(x_{n})}, \\[1.8ex]
%z_{n} = y_{n} - \dfrac{f(x_n)f(y_n)}{(f(x_n)-f(y_n))^2} \dfrac{f(x_n)}{f'(x_n)}, \\[1.8ex]
%x_{n+1} = z_{n} - \dfrac{f(z_n)}{f'(x_n)}
%\left( \dfrac{1+(\frac{i+1}{2})t_n+(i+3) u_n+(i+2)t_n^2+4(\frac{i+1}{2})
%t_n^3}{1+(\frac{i-3}{2})t_n+(i+1)u_n+ t_n^2} \cdot
%\dfrac{1+(\frac{i-1}{2}) s_n}{1+(\frac{i-3}{2})s_n} \right)
%\end{cases}
\left\{
\begin{aligned}
y_{n} & := x_{n} - \dfrac{f(x_{n})}{f'(x_{n})}, \\[1.8ex]
z_{n} & := y_{n} - \dfrac{f(x_n)f(y_n)}{(f(x_n)-f(y_n))^2} \dfrac{f(x_n)}{f'(x_n)}, \\[1.8ex]
x_{n+1} & := z_{n} - \dfrac{f(z_n)}{f'(x_n)}
\left( \dfrac{1+(\frac{i+1}{2})t_n+(i+3) u_n+(i+2)t_n^2+4(\frac{i+1}{2})
t_n^3}{1+(\frac{i-3}{2})t_n+(i+1)u_n+ t_n^2} \cdot
\dfrac{1+(\frac{i-1}{2}) s_n}{1+(\frac{i-3}{2})s_n} \right)
\end{aligned}
\right.
\end{equation}
with $t_n = \frac{f(y_n)}{f(x_n)}$, $u_n = \frac{f(z_n)}{f(x_n)}$,
$s_n = \frac{f(z_n)}{f(y_n)}$.

We will apply in the next sections the new methods~\eqref{n1}
and~\eqref{n2} to several benchmark examples and will
compare the new methods with some existing optimal three-point methods
of order eight having the same optimal computational efficiency index equal
to $\sqrt[4]{8} \simeq 1.68179$, see~\cite{Ostrowski, Traub}.

The existing methods that we are going to use to compare are the following:

\paragraph{Method 3:}
The method by Chun and Lee \cite{Chun1} is given by
\begin{equation}
\label{o1}
%\begin{cases}
%y_{n} = x_{n} - \dfrac{f(x_{n})}{f'(x_{n})}, \\[1.8ex]
%z_{n} = y_{n} - \dfrac{f(y_n)}{f'(x_n)}
%\cdot \dfrac{1}{\left(1-\frac{f(y_n)}{f(x_n)}\right)^2}, \\[1.8ex]
%x_{n+1} = z_{n} - \dfrac{f(z_n)}{f'(x_n)}
%\cdot \dfrac{1}{\left(1-H(t_n)-J(s_n)-P(u_n)\right)^2}
%\end{cases}
\left\{
\begin{aligned}
y_{n} & := x_{n} - \dfrac{f(x_{n})}{f'(x_{n})}, \\[1.8ex]
z_{n} & := y_{n} - \dfrac{f(y_n)}{f'(x_n)}
\cdot \dfrac{1}{\left(1-\frac{f(y_n)}{f(x_n)}\right)^2}, \\[1.8ex]
x_{n+1} & := z_{n} - \dfrac{f(z_n)}{f'(x_n)}
\cdot \dfrac{1}{\left(1-H(t_n)-J(s_n)-P(u_n)\right)^2}
\end{aligned}
\right.
\end{equation}
with weight functions
\begin{equation*}
%\label{t1}
H(t_n) = -\beta-\gamma+t_n+\frac{t_n^2}{2}-\frac{t_n^3}{2}, \quad
J(s_n) = \beta+\frac{s_n}{2}, \quad P(u_n) = \gamma+\frac{u_n}{2},
\end{equation*}
where $t_n = \frac{f(y_n)}{f(x_n)}$, $s_n = \frac{f(z_n)}{f(x_n)}$,
$u_n = \frac{f(z_n)}{f(y_n)}$, and $\beta, \gamma \in \mathbb{R}$.
Note that the parameters $\beta$ and $\gamma$ cancel when used
in~\eqref{o1}. Hence, their choice has no contribution to the method.

\paragraph{Method 4:}
The method by B. Neta~\cite{Neta0}, see also~\cite[formula~(9)]{Neta1}, is
given by
\begin{equation}
\label{o2}
%\begin{cases}
%y_{n} = x_{n} - \dfrac{f(x_n)}{f'(x_n)}, \\[1.8ex]
%z_{n} = y_{n} - \dfrac{f(x_n)+Af(y_n)}{f(x_n)+(A-2)f(y_n)}\cdot
%\dfrac{f(y_n)}{f'(x_n)},
%\quad A\in \mathbb{R},\\[1.8ex]
%x_{n+1} = y_n + \delta_1f^2(x_n)+\delta_2f^3(x_n),
%\end{cases}
\left\{
\begin{aligned}
y_{n} & := x_{n} - \dfrac{f(x_n)}{f'(x_n)}, \\[1.8ex]
z_{n} & := y_{n} - \dfrac{f(x_n)+Af(y_n)}{f(x_n)+(A-2)f(y_n)}\cdot
\dfrac{f(y_n)}{f'(x_n)},
\quad A\in \mathbb{R},\\[1.8ex]
x_{n+1} & := y_n + \delta_1f^2(x_n)+\delta_2f^3(x_n),
\end{aligned}
\right.
\end{equation}
where
%\begin{equation*}
%F_y = f(y_n)-f(x_n), \quad F_z = f(z_n)-f(x_n),
%\end{equation*}
%\begin{equation*}
%\zeta_y = \dfrac{1}{F_y}
%\left(\dfrac{y_n-x_n}{F_y}-\dfrac{1}{f'(x_n)}\right),
%\quad
%\zeta_z = \dfrac{1}{F_z}
%\left(\dfrac{z_n-x_n}{F_z}-\dfrac{1}{f'(x_n)}\right),
%\end{equation*}
%\begin{equation*}
%\delta_2 = -\dfrac{\zeta_y-\zeta_z}{F_y-F_z},
%\quad \quad
%\delta_1 = \zeta_y+\delta_2F_y.
%\end{equation*}
\begin{alignat*}{2}
F_y & = f(y_n)-f(x_n), &\qquad F_z &= f(z_n)-f(x_n),\\
\zeta_y & = \dfrac{1}{F_y}
\left(\dfrac{y_n-x_n}{F_y}-\dfrac{1}{f'(x_n)}\right),
&\qquad
\zeta_z & = \dfrac{1}{F_z}
\left(\dfrac{z_n-x_n}{F_z}-\dfrac{1}{f'(x_n)}\right),\\
\delta_2 & = -\dfrac{\zeta_y-\zeta_z}{F_y-F_z},
&\qquad
\delta_1 & = \zeta_y+\delta_2F_y.
\end{alignat*}
We will use $A = 0$ in the numerical experiments of this paper.

\paragraph{Method 5:}
The Sharma and Sharma method \cite{Sharma1} is given by
\begin{equation}
\label{o3}
%\begin{cases}
%y_{n} = x_{n} - \dfrac{f(x_n)}{f'(x_n)}, \\[1.8ex]
%z_{n} = y_{n} - \dfrac{f(y_n)}{f'(x_n)} \cdot \dfrac{f(x_n)}{f(x_n)-2f(y_n)}, \\[1.8ex]
%x_{n+1} = z_{n} - \dfrac{f[x_n,y_n]f(z_n)}{f[x_n,z_n]f[y_n,z_n]}\,W(t_n),
%\end{cases}
\left\{
\begin{aligned}
y_{n} & := x_{n} - \dfrac{f(x_n)}{f'(x_n)}, \\[1.8ex]
z_{n} & := y_{n} - \dfrac{f(y_n)}{f'(x_n)} \cdot \dfrac{f(x_n)}{f(x_n)-2f(y_n)}, \\[1.8ex]
x_{n+1} & := z_{n} - \dfrac{f[x_n,y_n]f(z_n)}{f[x_n,z_n]f[y_n,z_n]}\,W(t_n),
\end{aligned}
\right.
\end{equation}
with the weight function
\begin{equation*}
%\label{d6}
W(t_n) = 1+\frac{t_n}{1+\alpha t_n},
\quad \alpha\in \mathbb{R},
\end{equation*}
and $t_n = \frac{f(z_n)}{f(x_n)}$.
We will use $\alpha=1$ in the numerical experiments of this paper.

\paragraph{Method 6:}
The method from Babajee, Cordero, Soleymani and Torregrosa \cite{Babajee}
is given by
\begin{equation}
\label{o4}
%\begin{cases}
%y_{n} = x_{n} - \dfrac{f(x_n)}{f'(x_n)}
%\left( 1 + \left(\dfrac{f(x_n)}{f'(x_n)}\right)^5 \right), \\[1.8ex]
%z_{n} = y_{n} - \dfrac{f(y_n)}{f'(x_n)}
%\left( 1 - \dfrac{f(y_n)}{f(x_n)}\right)^{-2}, \\[1.8ex]
%x_{n+1} = z_n - \dfrac{f(z_n)}{f'(x_n)}
%\cdot \dfrac{1+\left(\dfrac{f(y_n)}{f(x_n)}\right)^2 + 5\left(\dfrac{f(y_n)}{f(x_n)}\right)^4
%+ \dfrac{f(z_n)}{f(y_n)}}
%{\left( 1 - \dfrac{f(y_n)}{f(x_n)} - \dfrac{f(z_n)}{f(x_n)} \right)^2} .
%\end{cases}
\left\{
\begin{aligned}
y_{n} & := x_{n} - \dfrac{f(x_n)}{f'(x_n)}
\left( 1 + \left(\dfrac{f(x_n)}{f'(x_n)}\right)^5 \right), \\[1.8ex]
z_{n} & := y_{n} - \dfrac{f(y_n)}{f'(x_n)}
\left( 1 - \dfrac{f(y_n)}{f(x_n)}\right)^{-2}, \\[1.8ex]
x_{n+1} & := z_n - \dfrac{f(z_n)}{f'(x_n)}
\cdot \dfrac{1+\left(\dfrac{f(y_n)}{f(x_n)}\right)^2 + 5\left(\dfrac{f(y_n)}{f(x_n)}\right)^4
+ \dfrac{f(z_n)}{f(y_n)}}
{\left( 1 - \dfrac{f(y_n)}{f(x_n)} - \dfrac{f(z_n)}{f(x_n)} \right)^2} .
\end{aligned}
\right.
\end{equation}

%--------------------------------------------------------------%
\section{Numerical examples}
\label{sec:examples}
%--------------------------------------------------------------%

The particular cases~\eqref{n1} and~\eqref{n2} of the the three-point
method~\eqref{a3} are tested on a number of nonlinear
equations. To obtain a high accuracy and avoid the loss of
significant digits, we employed multi-precision arithmetic with
20\,000 significant decimal digits in the programming package
Mathematica.
% (we have used version~8).
% \cite{Hazrat}.

In order to test our proposed methods \eqref{n1} and \eqref{n2},
and also to compare them with the methods \eqref{o1}, \eqref{o2},
\eqref{o3}, and \eqref{o4}, we compute the error, the computational
order of convergence (COC) by the approximate formula~\cite{coc}
\begin{equation}
\label{coc}
\textup{COC} \approx \frac{\ln|(x_{n+1}-x^{*})/(x_{n}-x^{*})|}{\ln|(x_{n}-x^{*})/(x_{n-1}-x^{*})|},
\end{equation}
and the approximated computational order of convergence (ACOC) by
the formula~\cite{acoc}
\begin{equation}
\label{acoc}
\textup{ACOC} \approx \frac{\ln|(x_{n+1}-x_{n})/(x_{n}-x_{n-1})|}{\ln|(x_{n}-x_{n-1})/(x_{n-1}-x_{n-2})|}.
\end{equation}

It is worth noting that COC has been used in the recent years.
Nevertheless, ACOC is more practical because it does not require to know the
root~$x^{*}$. For a comparison among several convergence orders,
see~\cite{acoc-et-al}. Moreover, we should note that the results
for these formulas not always coincide with or approximate the
exact convergence order of the method when they are applied to a
particular example. The reason is that we have in the error equations of
the methods some coefficients that depend on $c_k$ (see
the proof of Theorem~\ref{thm:main}). Hence, these $c_k$'s may vanish
or vary for different kinds of examples. But, in general, a
``random'' example should provide good approximations for the
order of convergence of the method.

On the other hand, it is nice to note that, given an iterative method,
computing COC or ACOC on several examples is a good experiment to check
theoretical errors in the deduction of the method and to check practical
errors in the implementation of the method in a computer.
For a general problem it will be difficult that COC or ACOC approach
the theoretical order of convergence by chance.

We have used both COC and ACOC for checking the accuracy of the considered
methods. Note that both COC and ACOC give already for small values of $n$
good experimental approximations to convergence order.

\begin{table}[htb!]
\begin{center}
\begin{tabular}{lcc}
   \toprule
    test function $f_j$ & root $x^*$ & initial guess $x_0$
  \\ \midrule
    $f_1(x) = \ln (1+x^2)+e^{x^2-3x}\sin x$ & $0$ & $0.35$
  \\[0.5ex]
   $f_2(x) = 1+e^{2+x-x^2}+x^3-\cos(1+x)$ & $-1$ & $-0.3$
  \\[0.5ex]
   $f_3(x) = (1+x^2)\cos\frac{\pi x}{2}+\frac{\ln(x^2+2x+2)}{1+x^2}$ & $-1$ & $-1.1$
  \\[0.5ex]
   $f_4(x) = x^4+\sin\frac{\pi}{x^2}-5$ & $\sqrt{2}$ & $1.5$
  \\[0.5ex] 
  \bottomrule
\end{tabular}
\caption{Test functions $f_1, \dots, f_4$, root
$x^*$, and initial guess $x_0$.}
\label{table1}
\end{center}
\end{table}

In what follows, we are going to perform this kind of numerical experiments with the four test functions
$f_j(x)$, $j=1,\dots,4$, that appear in Table~\ref{table1}. In every case, and using the six
eighth-order iterative methods described in the paper, we are going to reach the root $x^*$ starting in the point~$x_0$.

%In Table~\ref{table2} our new three-point methods \eqref{e1}, \eqref{e2} and
%\eqref{e3} are tested on the four nonlinear equations $f_j(x)=0$, $j=1,2,3,4$.
%Notice that, to estimate the COC and the ACOC, it has been enough to use
%$n=3$ or $4$ in \eqref{coc} and \eqref{acoc} to get excellent approximations
%of the order of convergence.

In Table~\ref{table2-3}, our new three-point methods \eqref{n1} and
\eqref{n2} are tested on the four nonlinear equations $f_j(x)=0$,
$j=1,2,3,4$, and compared them with the methods~\eqref{o1},
\eqref{o2}, \eqref{o3}, and~\eqref{o4} from other authors.
We abbreviate \eqref{n1}--\eqref{o4} as M1--M6. Notice that, to
estimate the COC and the ACOC, it has been enough to use $n=3$
in \eqref{coc} and \eqref{acoc} to get excellent
approximations of the order of convergence.

\begin{table}[htb!]\small
\begin{center}
\begin{tabular}{cllllll}
 \toprule
  & \multicolumn{1}{c}{M1} & \multicolumn{1}{c}{M2} & \multicolumn{1}{c}{M3}
  & \multicolumn{1}{c}{M4} & \multicolumn{1}{c}{M5} & \multicolumn{1}{c}{M6} \\
\midrule
$f_1$, $x_0=0.35$\\
$|x_{1}-x^{*}|$ & $0.140\mathrm{e}{-}3$ &  $0.318\mathrm{e}{-}3$ & $0.721\mathrm{e}{-}4$
& $0.893\mathrm{e}{-}4$ & $0.753\mathrm{e}{-}4$ & $0.347\mathrm{e}{-}3$ \\
$|x_{2}-x^{*}|$ & $0.583\mathrm{e}{-}28$& $0.562\mathrm{e}{-}25$ & $0.230\mathrm{e}{-}30$
& $0.126\mathrm{e}{-}30$ & $0.619\mathrm{e}{-}31$ & $0.471\mathrm{e}{-}25$ \\
$|x_{3}-x^{*}|$ & $0.362\mathrm{e}{-}223$ & $0.531\mathrm{e}{-}199$ &  $0.252\mathrm{e}{-}242$
& $0.200\mathrm{e}{-}245$ & $0.128\mathrm{e}{-}247$ & $0.546\mathrm{e}{-}200$ \\
COC &  $8.0000$ & $8.0000$ &  $8.0000$ & $8.0000$ & $8.0000$ & $8.0000$ \\
ACOC & $7.9999$ & $8.0000$ &  $7.9999$ & $7.9999$ & $7.9999$ & $7.9999$ \\%[0.5ex]
\midrule
$f_2$, $x_0=-0.3$\\
$|x_{1}-x^{*}|$ & $0.526\mathrm{e}{-}4$ & $0.113\mathrm{e}{-}3$ & $0.157\mathrm{e}{-}3$
& $0.763\mathrm{e}{-}4$ & $0.871\mathrm{e}{-}4$ & $0.411\mathrm{e}{-}3$ \\
$|x_{2}-x^{*}|$ & $0.534\mathrm{e}{-}37$& $0.263\mathrm{e}{-}33$ & $0.119\mathrm{e}{-}33$
& $0.540\mathrm{e}{-}35$ & $0.134\mathrm{e}{-}34$ & $0.377\mathrm{e}{-}29$ \\
$|x_{3}-x^{*}|$ & $0.599\mathrm{e}{-}301$ & $0.226\mathrm{e}{-}270$ & $0.138\mathrm{e}{-}274$
& $0.342\mathrm{e}{-}284$ & $0.438\mathrm{e}{-}281$ & $0.189\mathrm{e}{-}237$ \\
COC &  $8.0000$ & $8.0000$ &  $8.0000$ & $8.0000$ & $8.0000$ & $8.0000$ \\
ACOC & $7.9999$ & $7.9999$ & $7.9998$ & $7.9999$ & $7.9999$ & $7.9999$ \\%[0.5ex]
\midrule
$f_3$, $x_0=-1.1$\\
$|x_{1}-x^{*}|$ & $0.235\mathrm{e}{-}7$ & $0.298\mathrm{e}{-}7$ & $0.614\mathrm{e}{-}8$
& $0.388\mathrm{e}{-}8$ & $0.175\mathrm{e}{-}8$ & $0.554\mathrm{e}{-}8$ \\
$|x_{2}-x^{*}|$ & $0.393\mathrm{e}{-}60$ & $0.373\mathrm{e}{-}59$ & $0.328\mathrm{e}{-}65$
& $0.254\mathrm{e}{-}67$ & $0.154\mathrm{e}{-}70$ & $0.426\mathrm{e}{-}66$ \\
$|x_{3}-x^{*}|$ & $0.239\mathrm{e}{-}482$ & $0.222\mathrm{e}{-}474$ & $0.217\mathrm{e}{-}523$
& $0.877\mathrm{e}{-}541$ & $0.5821\mathrm{e}{-}567$ & $0.528\mathrm{e}{-}531$ \\
COC &  $8.0000$ & $8.0000$ & $8.0000$ & $8.0000$ & $8.0000$ & $8.0000$ \\
ACOC & $7.9999$ & $7.9999$ & $8.0000$ & $7.9999$ & $8.0000$ & $8.0000$ \\%[0.5ex]
\midrule
$f_4$, $x_0=1.5$\\
$|x_{1}-x^{*}|$ & $0.286\mathrm{e}{-}8$ & $0.602\mathrm{e}{-}8$ & $0.433\mathrm{e}{-}8$
& $0.327\mathrm{e}{-}10$ & $0.642\mathrm{e}{-}10$ & $0.281\mathrm{e}{-}8$ \\
$|x_{2}-x^{*}|$ & $0.108\mathrm{e}{-}68$& $0.181\mathrm{e}{-}65$ & $0.134\mathrm{e}{-}66$
& $0.369\mathrm{e}{-}84$ & $0.101\mathrm{e}{-}81$ & $0.341\mathrm{e}{-}68$ \\
$|x_{3}-x^{*}|$ & $0.460\mathrm{e}{-}552$ & $0.121\mathrm{e}{-}525$ & $0.116\mathrm{e}{-}534$
& $0.967\mathrm{e}{-}676$ & $0.389\mathrm{e}{-}656$ & $0.161\mathrm{e}{-}547$ \\
COC &  $8.0000$ & $8.0000$ & $8.0000$ & $8.0000$ & $8.0000$ & $8.0000$ \\
ACOC & $8.0000$ & $8.0000$ & $7.9999$ & $7.9999$ & $7.9999$ & $8.0000$ \\
\bottomrule
\end{tabular}
\caption{Errors, COC, and ACOC for the iterative methods
\eqref{n1}--\eqref{o4} (abbreviated as M1--M6) applied to the find
the root of test functions $f_1,\dots,f_4$ given in
Table~\ref{table1}.}
\label{table2-3}
\end{center}
\end{table}

%\newpage

%--------------------------------------------------------------%
\section{Dynamic behavior}
\label{sec:dynamic}
%--------------------------------------------------------------%

We already observed that all methods converge if the initial guess
is chosen suitably. We now investigate the regions where we must
choose the initial point to achieve the root. In other words, we
numerically approximate the domain of attraction of the zeros as a
qualitative measure of how demanding is the method on the initial
approximation of the root. To answer the important question on the
dynamical behavior of the algorithms, we investigate the dynamics
of the new methods \eqref{n1} and~\eqref{n2} and compare
with common and well-performing methods from the literature,
namely~\eqref{o1}, \eqref{o2}, \eqref{o3}, and~\eqref{o4}. We recall in the
following line some basic concepts such as basin of attraction. For more
details and many other examples of the study of the dynamic behavior for
iterative methods, one can consult~\cite{Amat5, Babajee, Chicharro,
Ezquerro2, Ezquerro1, GutiMaVar, Paricio, Stewart, Varona}.

Let $Q:\mathbb{C} \to \mathbb{C}$ be a rational map on the
complex plane. For $z\in \mathbb{C} $, we define its orbit as the
set $\operatorname{orb}(z) = \{z,\,Q(z),\,Q^2(z),\dots\}$. A point $z_0 \in
\mathbb{C}$ is called periodic point with minimal period $m$ if
$Q^m(z_0) = z_0$ where $m$ is the smallest positive integer with this
property (and thus $\{z_0,Q(z_0),\dots, Q^{m-1}(z_0)\}$ is a cycle).
A periodic point with minimal period $1$ is called fixed
point. Moreover, a fixed point $z_0$ is called attracting if
$|Q'(z_0)|<1$, repelling if $|Q'(z_0)|>1$, and neutral
otherwise. The Julia set of a nonlinear map $Q(z)$, denoted by
$J(Q)$, is the closure of the set of its repelling periodic
points. The complement of $J(Q)$ is the Fatou set $F(Q)$.

In our case, the six methods \eqref{n1}--\eqref{o4} provide
iterative rational maps $Q(z)$ when they are applied to find
the roots of complex polynomials $p(z)$. In particular,
we are interesting in the basins of attraction of the roots
of the polynomials where the basin of attraction of a
root $z^*$ is the complex set
$\{ z_0 \in \mathbb{C} : \operatorname{orb}(z_0)  \to  z^* \}$.
It is well known that the basins of attraction of the different roots
lie in the Fatou set $F(Q)$.
The Julia set $J(Q)$ is, in general, a fractal and, in it,
the rational map $Q$ is unstable.

For the dynamical and graphical point of view, we take a
$512 \times 512$ grid of the square $[-3,3]\times[-3,3] \subset \mathbb{C}$
and assign a color to each point $z_0\in D$ according to the simple
root to which the corresponding orbit of the iterative method
starting from $z_0$ converges, and we mark the point as black if
the orbit does not converge to a root in the sense that after at
most $15$ iterations it has a distance to any of the roots which
is larger than $10^{-3}$. We have used only $15$ iterations because
we are using eighth-order methods so, if the method converges,
it is usually very fast. In this way, we distinguish the
attraction basins by their color.

\begin{table}[htb!]
\centering
\begin{tabular}{l @{\qquad} c}
   \toprule
    Test polynomials & Roots
   \\
   \midrule
   $p_1(z) = z^2-1$ & $1,\quad -1$
  \\[0.5ex]
   $p_2(z) = z^3-z $ & $0,\quad 1,\quad -1$
  \\[0.5ex]
   $p_3(z) = z(z^2+1)(z^2+4) $ & $0, \quad 2i, \quad -2i,\quad i, \quad -i$
  \\[0.5ex]
   $p_4(z) = (z^4-1)(z^2+2i)$ & $1, \quad i, \quad -1,\quad -i, \quad -1+i, \quad 1-i$
  \\[0.5ex]
   $p_5(z) = z^7-1$ & $e^{2k\pi i/7}, \qquad k=0,\dots,6$
  \\[0.5ex]
   $p_6(z) = (10z^5-1)(z^5+10)$ & ${(\tfrac{1}{10})}^{1/5} e^{2k\pi i/5},
         \quad (-10)^{1/5} e^{2k\pi i/5}, \qquad k=0,\dots,4$
  \\[0.5ex]
  \bottomrule
\end{tabular}
\caption{Test polynomials $p_1(z),\dots, p_6(z)$ and their roots.}
\label{table4}
\end{table}

Different colors are used for different roots. In the basins of attraction,
the number of iterations needed to achieve the root is shown by the
brightness. Brighter color means less iteration steps. Note that black
color denotes lack of convergence to any of the roots. This happens, in
particular, when the method converges to a fixed point that is not a root or
if it ends in a periodic cycle or at infinity. Actually and although we have
not done it in this paper, infinity can be considered an ordinary point if
we consider the Riemann sphere instead of the complex plane. In this case,
we can assign a new ``ordinary color'' for the basin of attraction of
infinity. Details for this idea can be found in~\cite{Paricio}.

\begin{figure}
\centering
\includegraphics[width=0.3\textwidth]{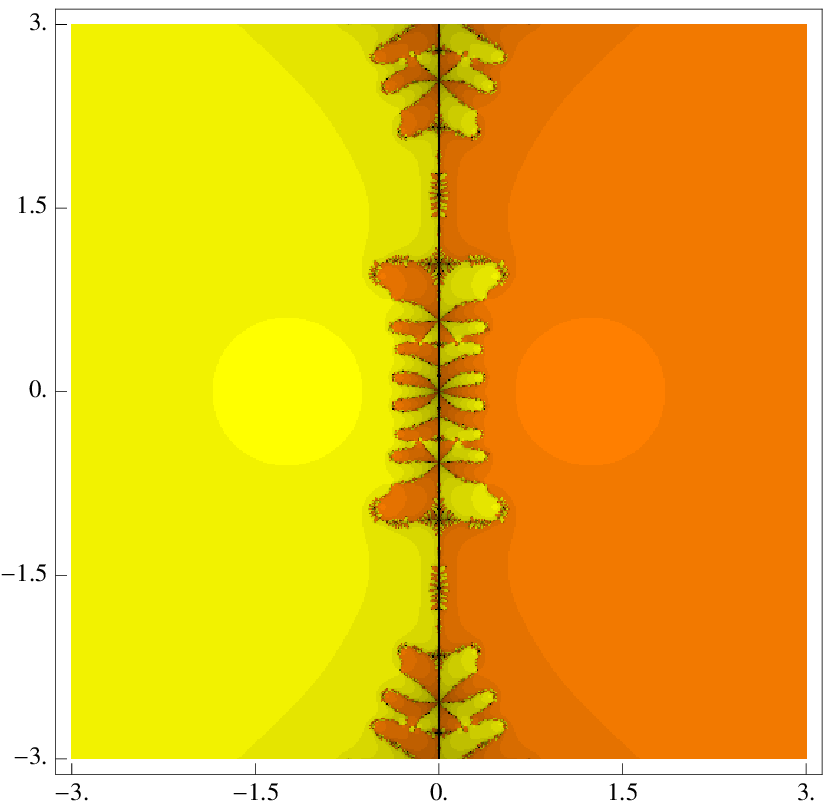}
\quad
\includegraphics[width=0.3\textwidth]{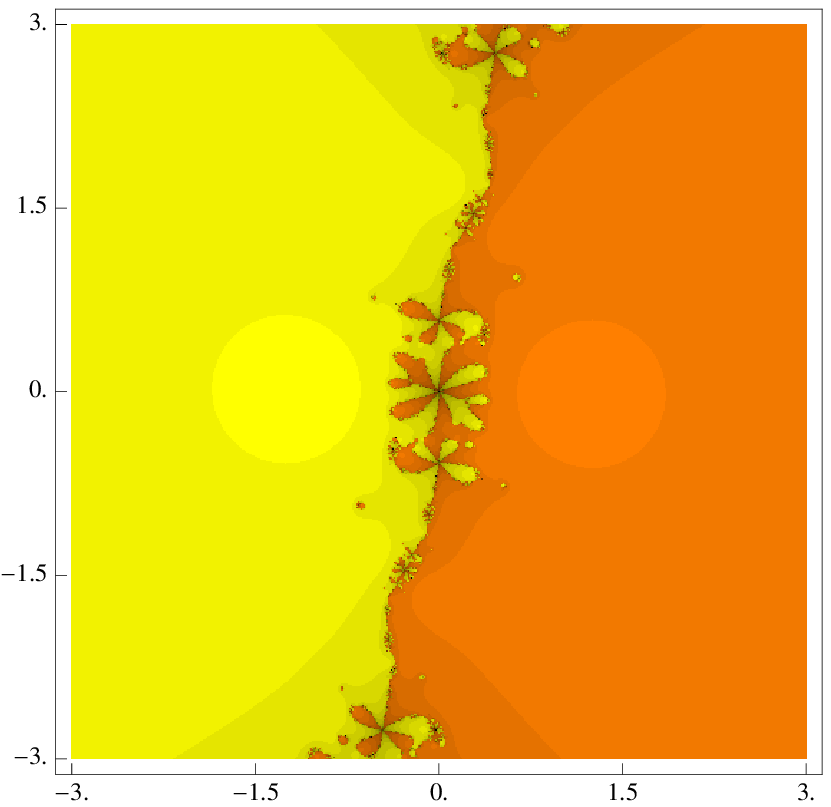}
\quad
\includegraphics[width=0.3\textwidth]{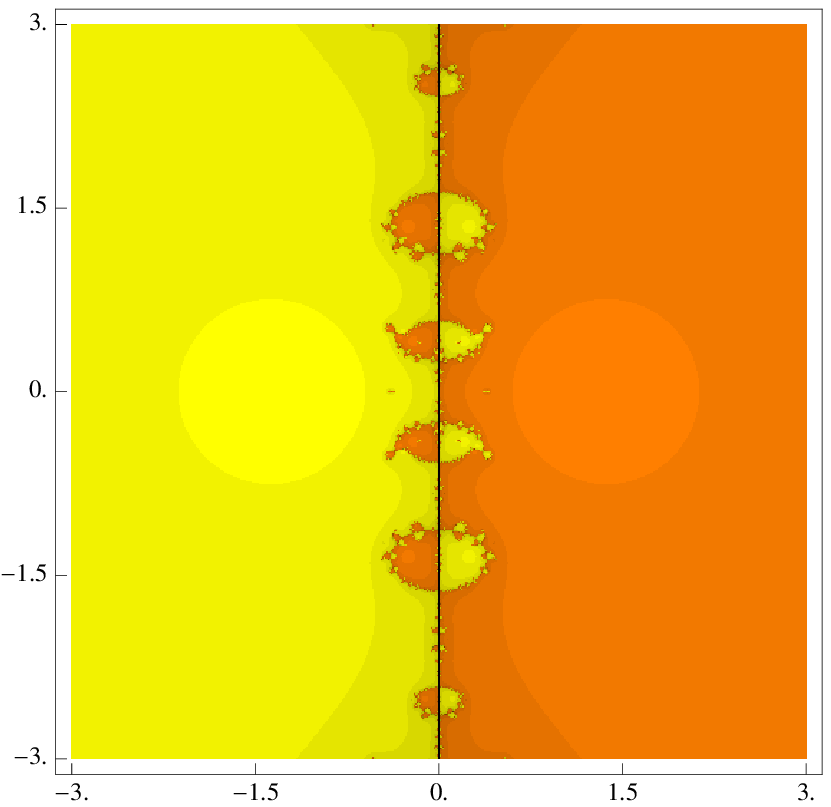}
\\
\includegraphics[width=0.3\textwidth]{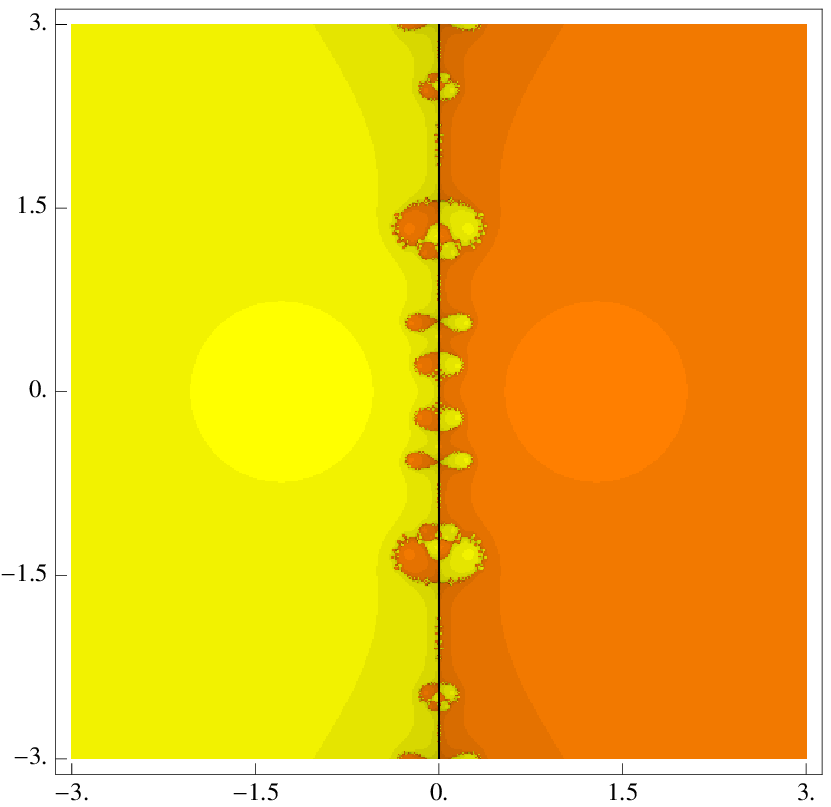}
\quad
\includegraphics[width=0.3\textwidth]{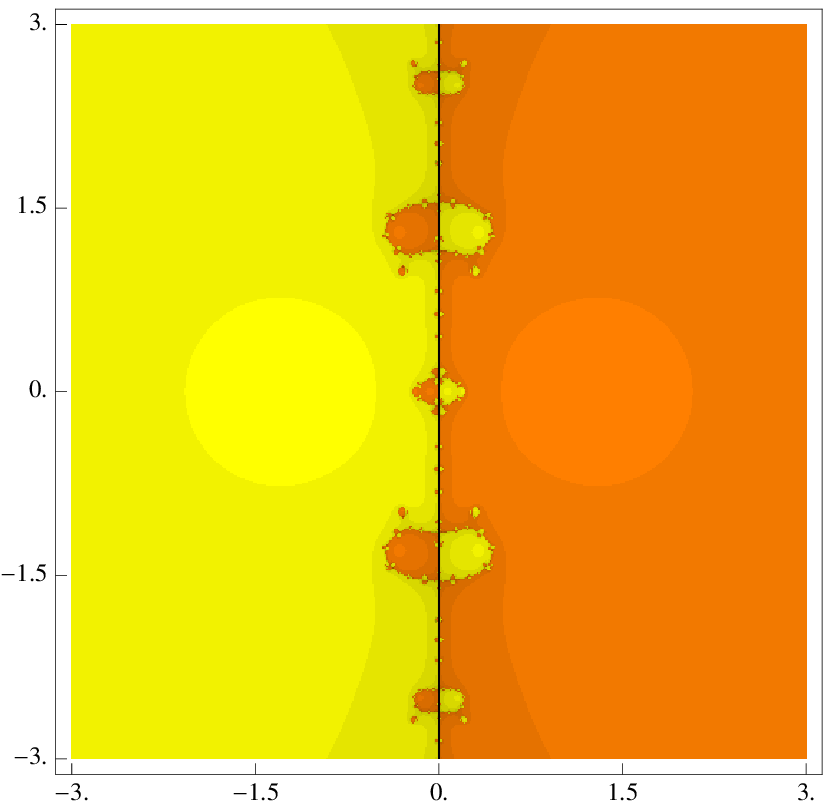}
\quad
\includegraphics[width=0.3\textwidth]{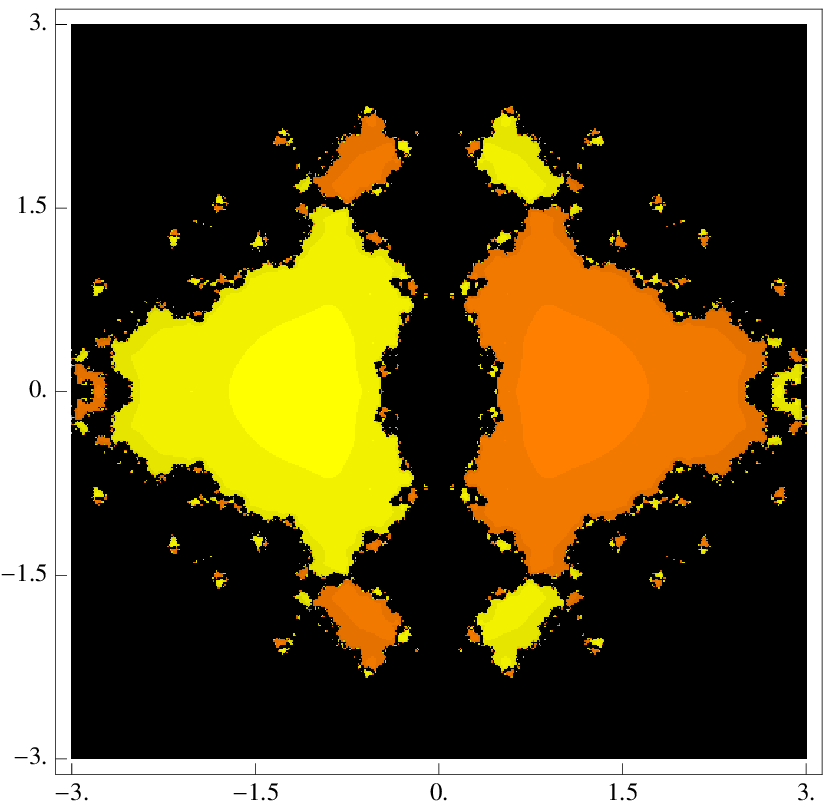}
\caption{Comparison of basins of attraction of methods
\eqref{n1}--\eqref{o4} for the test problem $p_1(z)= z^2-1=0$.}
\label{fig:figure1}
\end{figure}

\begin{figure}
\centering
\includegraphics[width=0.3\textwidth]{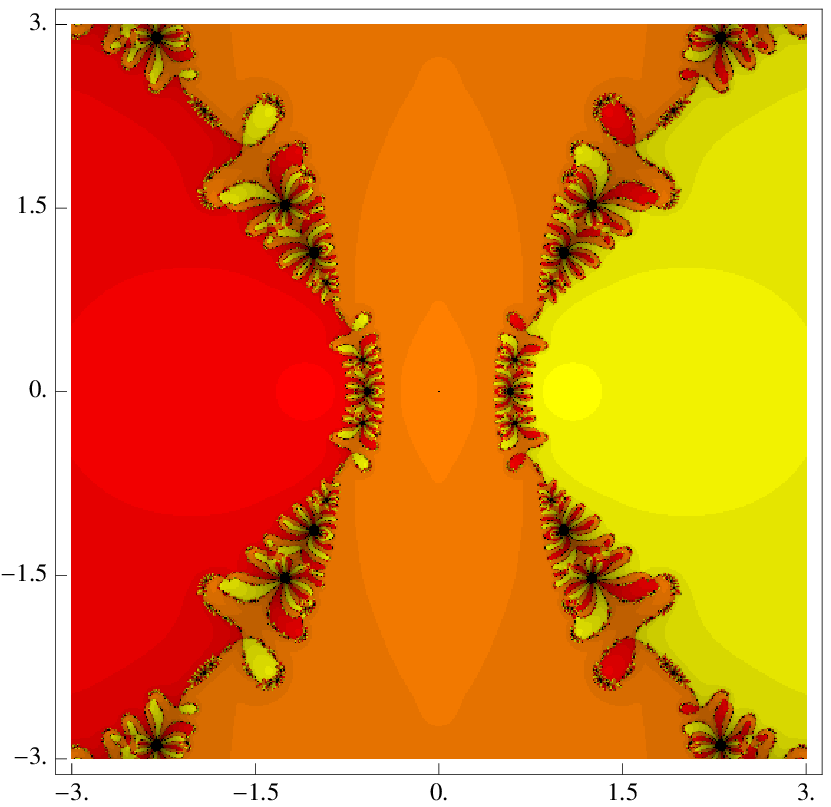}
\quad
\includegraphics[width=0.3\textwidth]{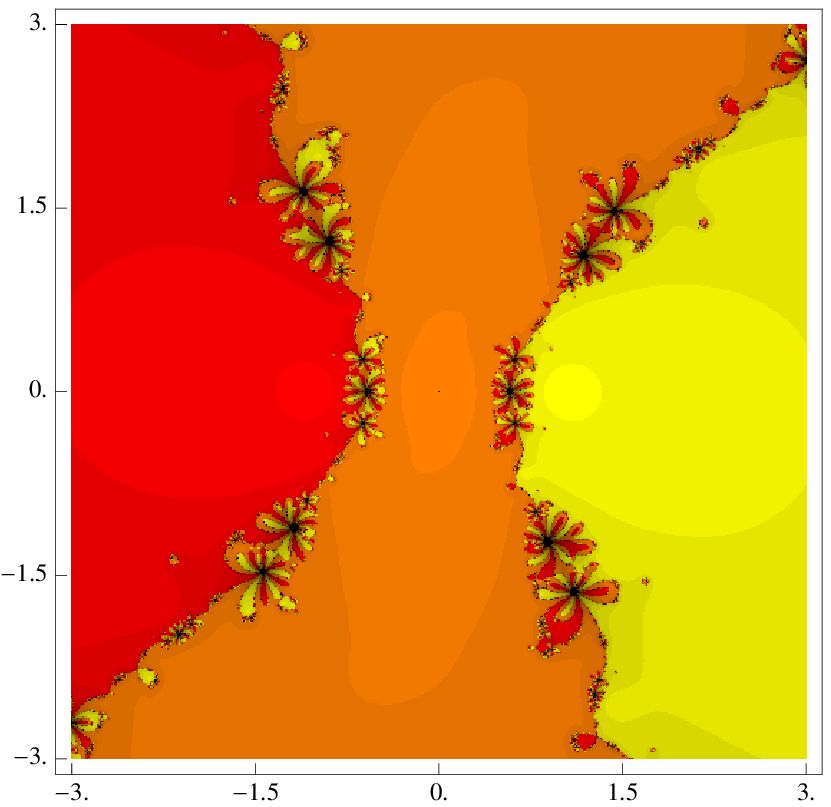}
\quad
\includegraphics[width=0.3\textwidth]{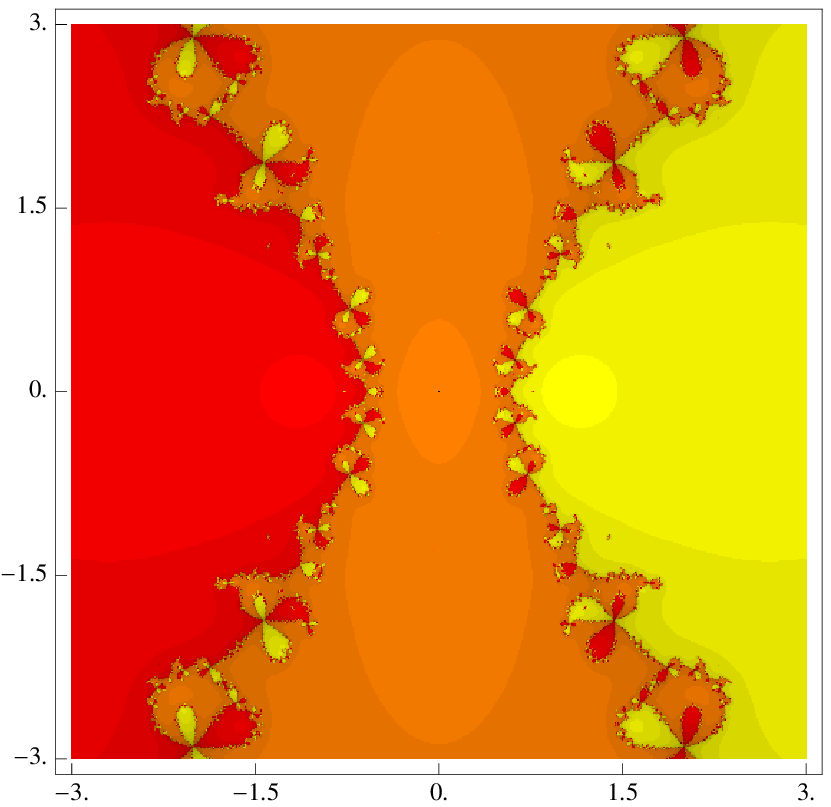}
\\
\includegraphics[width=0.3\textwidth]{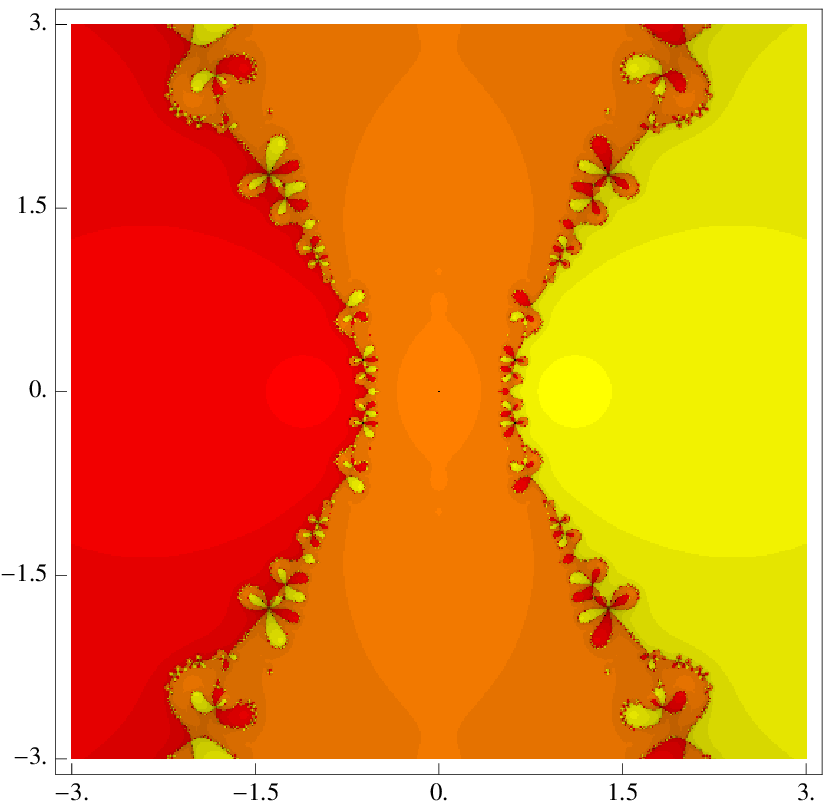}
\quad
\includegraphics[width=0.3\textwidth]{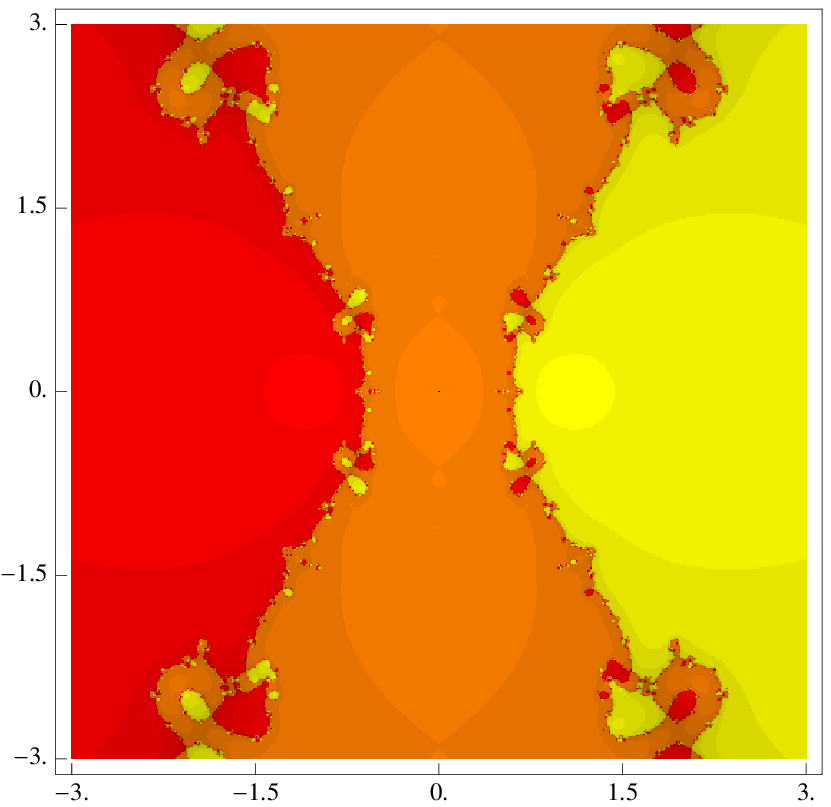}
\quad
\includegraphics[width=0.3\textwidth]{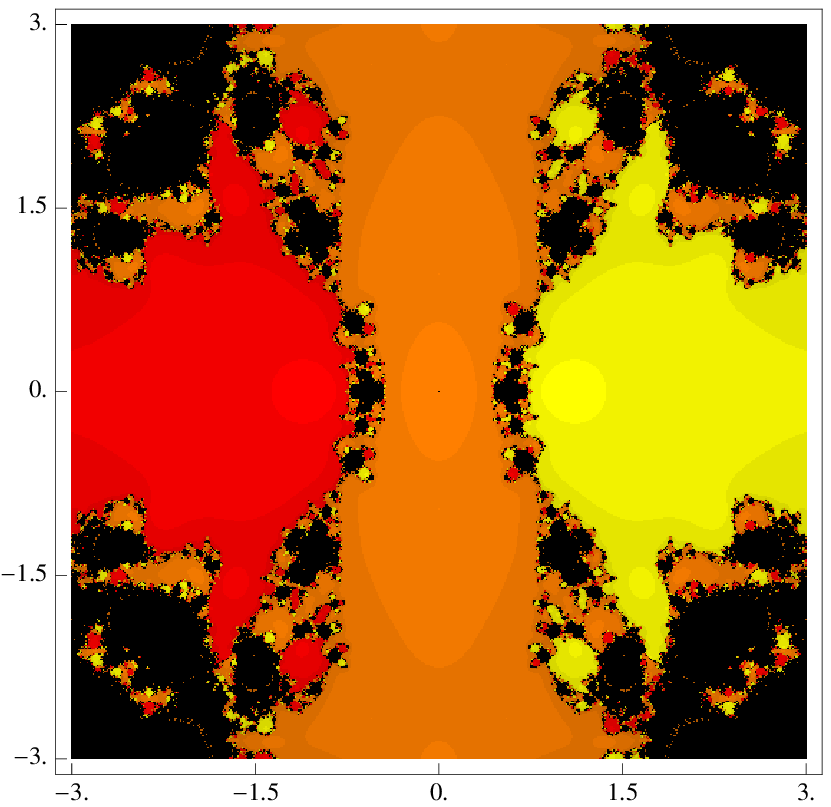}
\caption{Comparison of basins of attraction of methods
\eqref{n1}--\eqref{o4} for the test problem $p_2(z)= z^3-z=0$.}
\label{fig:figure2}
\end{figure}

Basins of attraction for the six methods \eqref{n1}--\eqref{o4}
for the six test problems $p_i(z)=0$, $i=1,\dots,6$, are illustrated
in Figures~\ref{fig:figure1}--\ref{fig:figure6} from left to right and
from top to bottom.
%$p_2(z)=0$ in Figure~\ref{fig:figure2}, for $p_3(z)=0$ in
%Figure~\ref{fig:figure3}, for $p_4(z)=0$ in
%Figure~\ref{fig:figure4}, for $p_5(z)=0$ in
%Figure~\ref{fig:figure5} and for $p_6(z)=0$ in
%Figure~\ref{fig:figure6}.

\begin{figure}
\centering
\includegraphics[width=0.3\textwidth]{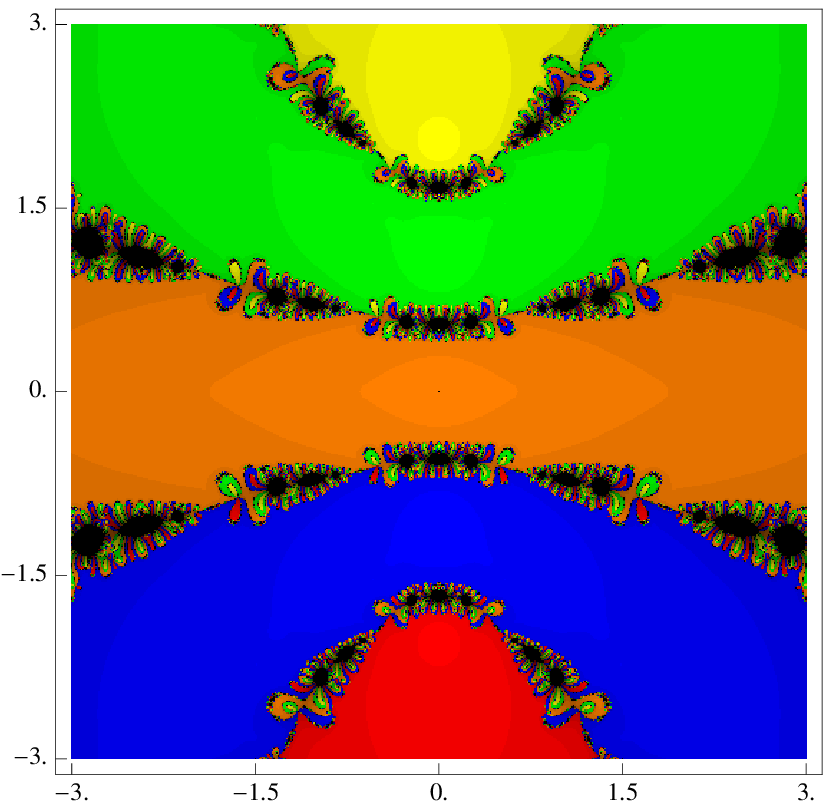}
\quad
\includegraphics[width=0.3\textwidth]{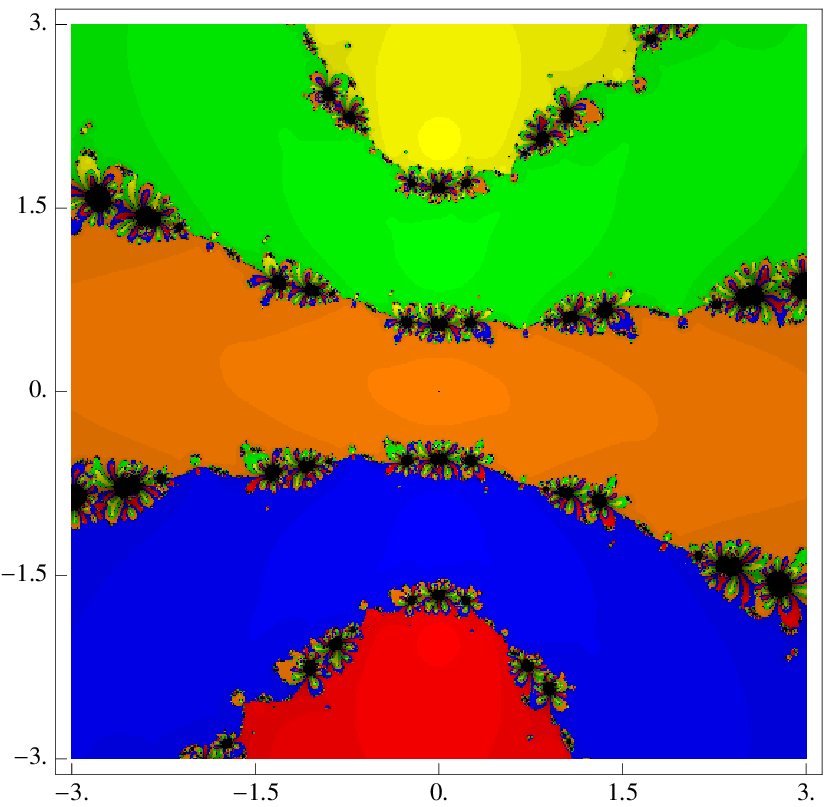}
\quad
\includegraphics[width=0.3\textwidth]{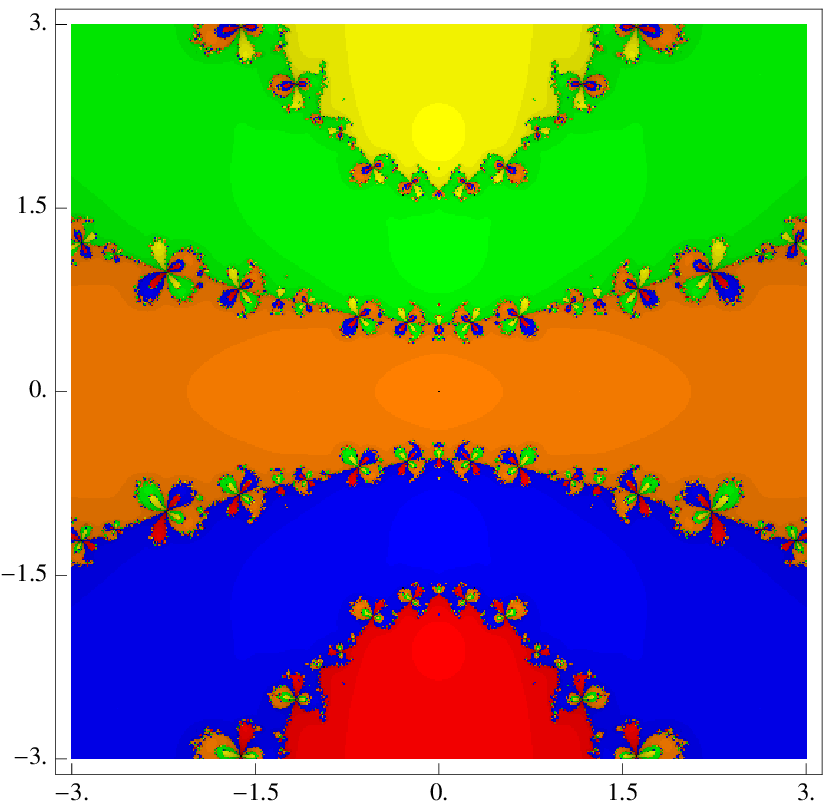}
\\
\includegraphics[width=0.3\textwidth]{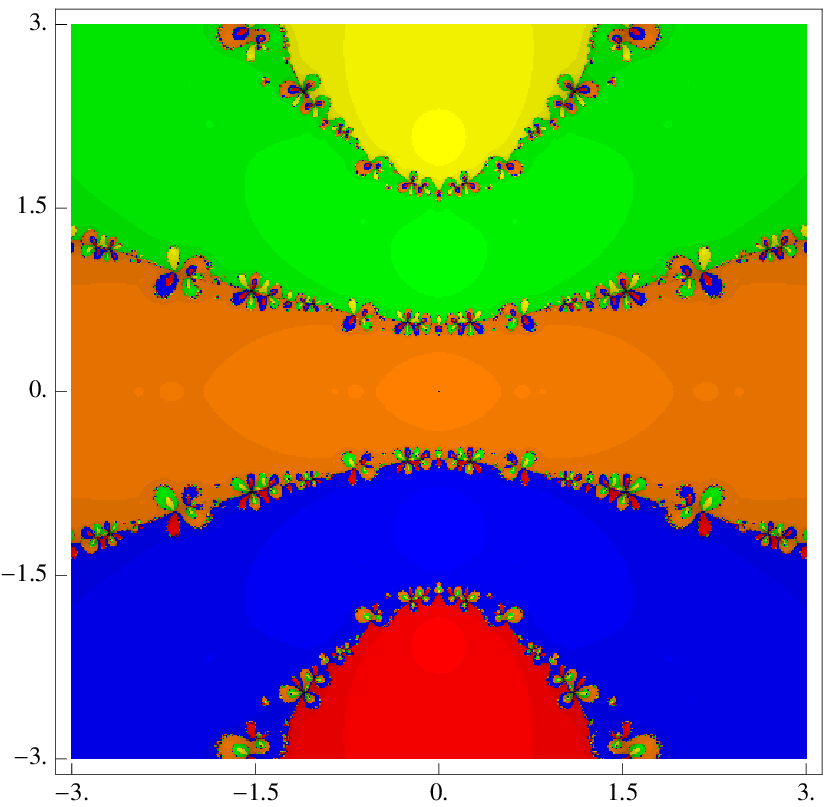}
\quad
\includegraphics[width=0.3\textwidth]{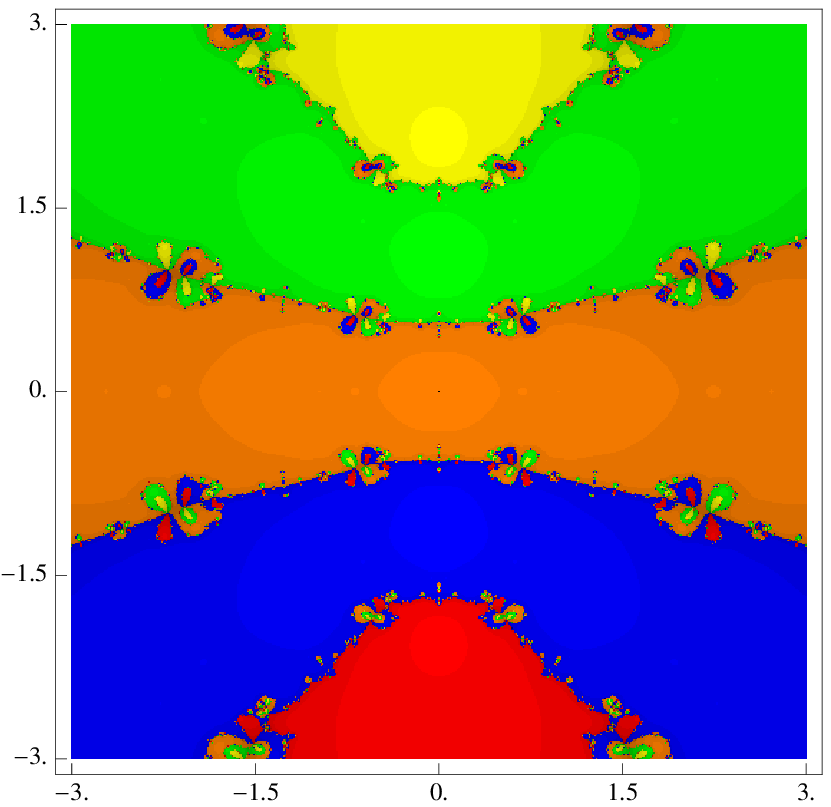}
\quad
\includegraphics[width=0.3\textwidth]{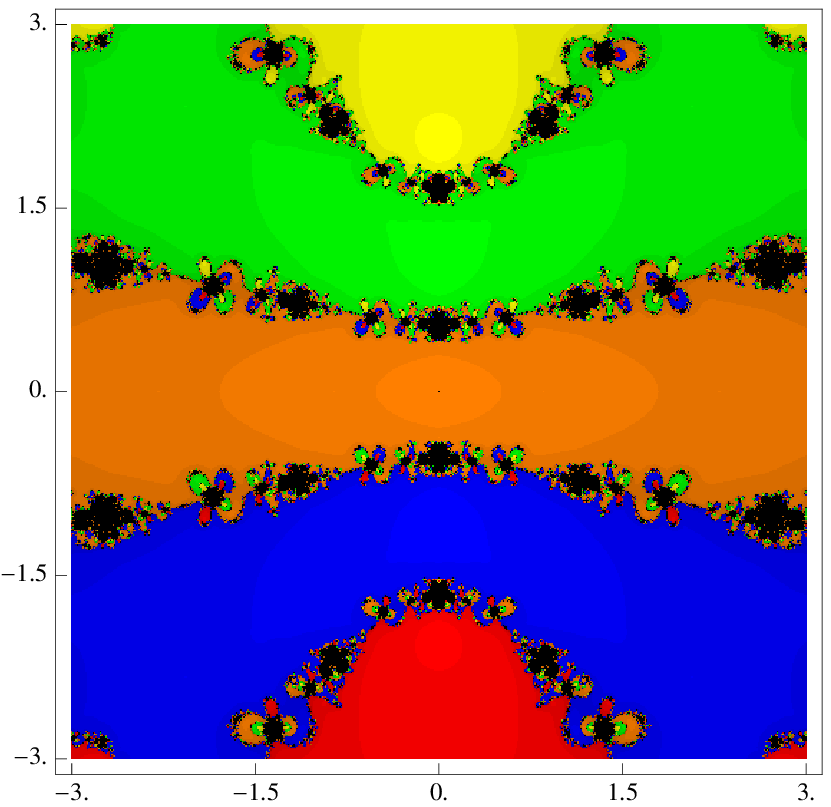}
\caption{Comparison of basins of attraction of methods
\eqref{n1}--\eqref{o4} for the test problem $p_3(z)=
z(z^2+1)(z^2+4)=0$.} 
\label{fig:figure3}
\end{figure}

\begin{figure}
\centering
\includegraphics[width=0.3\textwidth]{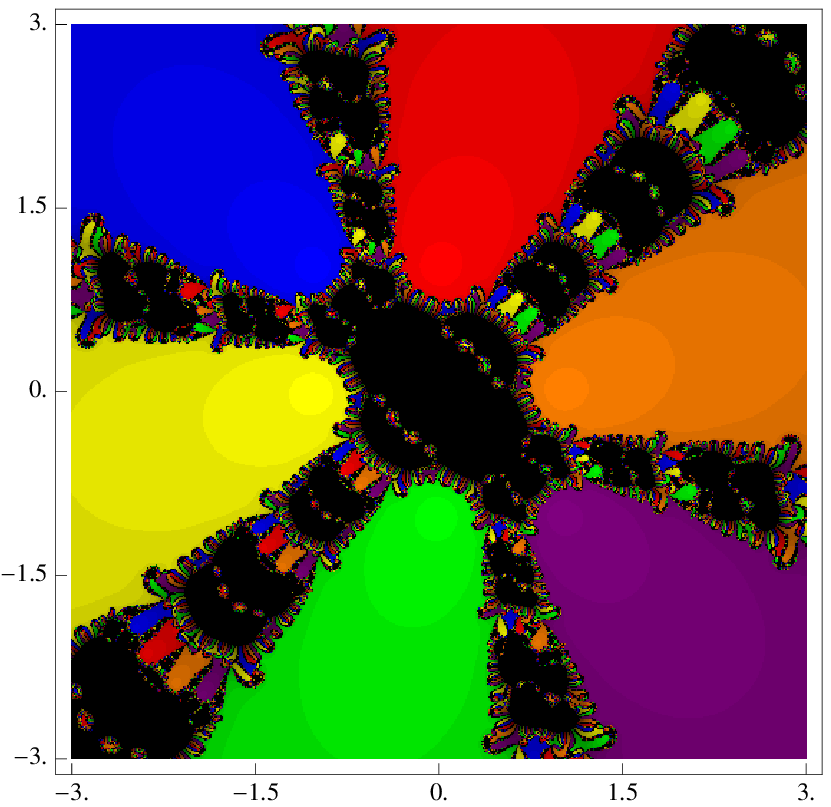}
\quad
\includegraphics[width=0.3\textwidth]{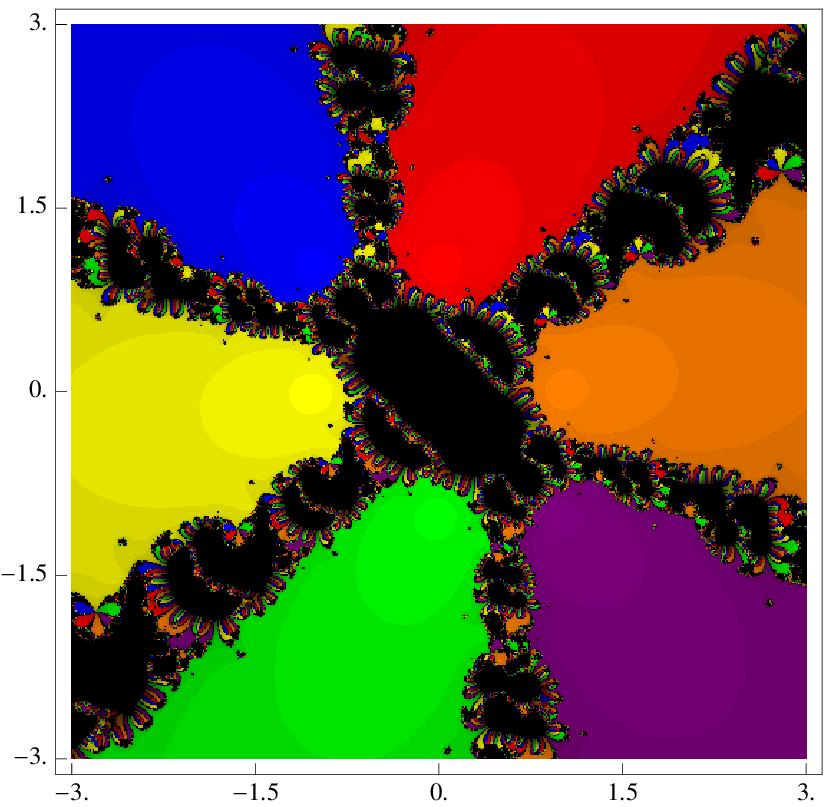}
\quad
\includegraphics[width=0.3\textwidth]{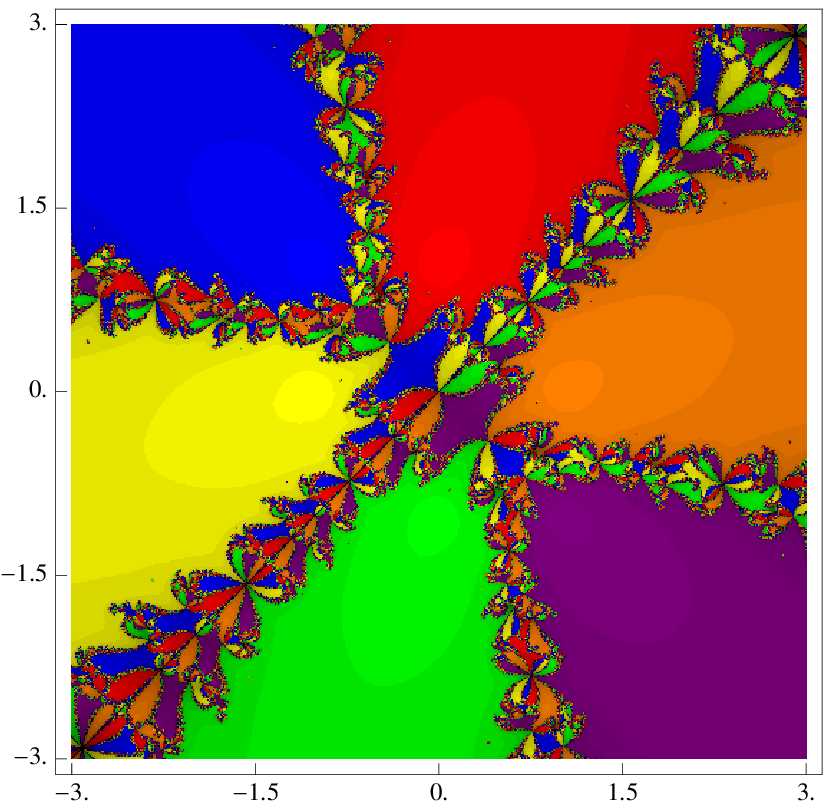}
\\
\includegraphics[width=0.3\textwidth]{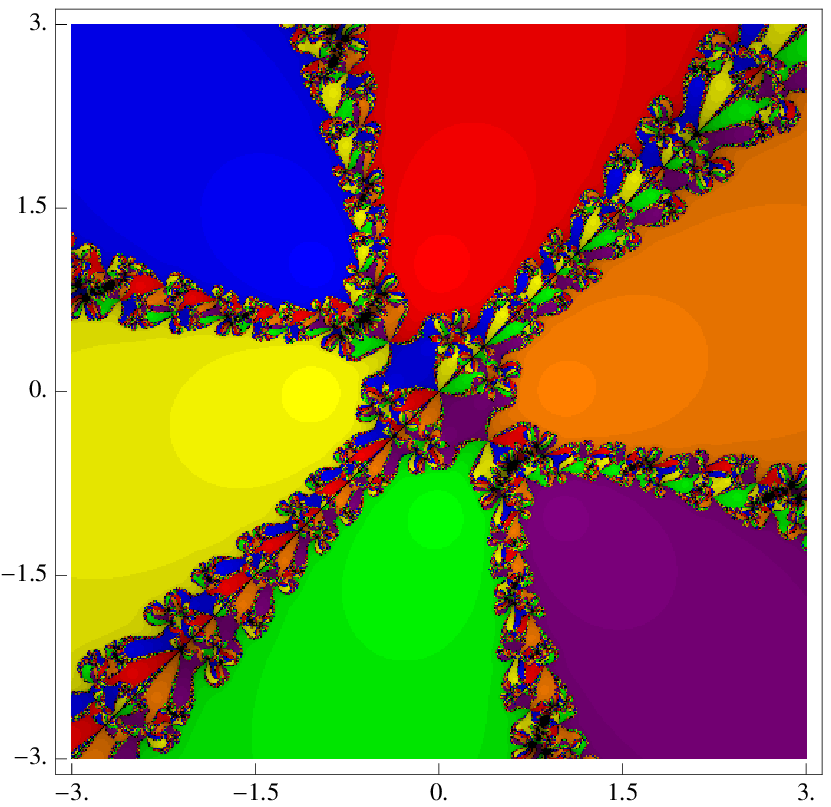}
\quad
\includegraphics[width=0.3\textwidth]{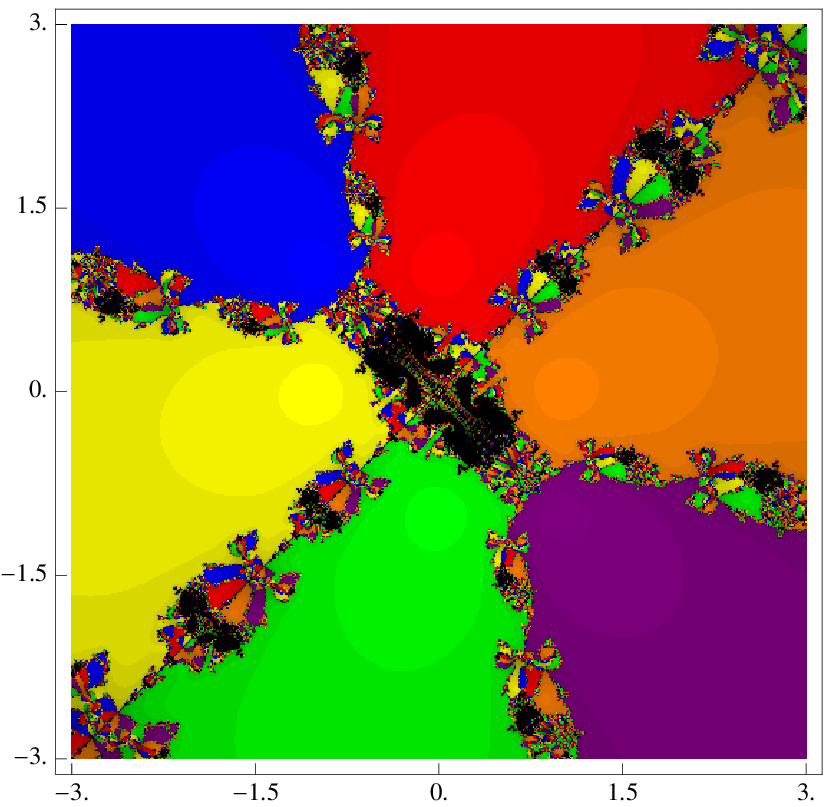}
\quad
\includegraphics[width=0.3\textwidth]{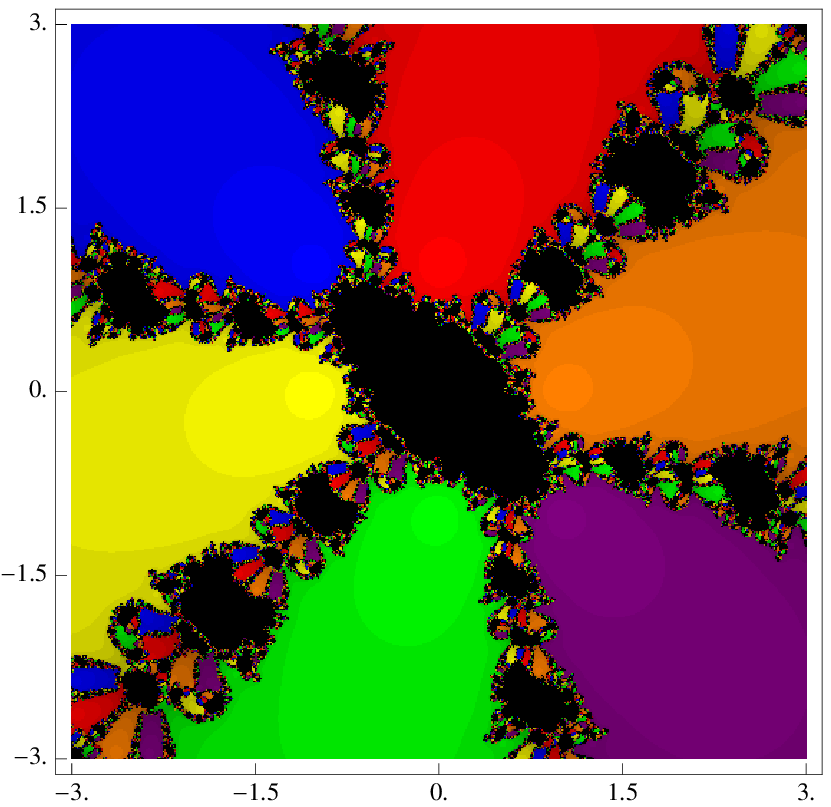}
\caption{Comparison of basins of attraction of methods
\eqref{n1}--\eqref{o4} for the test problem $p_4(z)=
(z^4-1)(z^2+2i)=0$.} 
\label{fig:figure4}
\end{figure}

From the pictures, we can easily judge the behavior and suitability
of any method depending on the circumstances.
% This is good entertainment.
If we choose an initial point $z_0$
in a zone where different basins of attraction touch each other, it is
impossible to predict which root is going to be reached by the
iterative method that starts in~$z_0$. Hence, it is
not a good choice. Both the black zones and the zones with a lot
of colors are not suitable to take the initial guess $z_0$ when we
want to achieve a precise root. The most attractive pictures appear
when we have very intricate frontiers between basins of attraction and
they correspond to the cases where the method is more demanding with
respect to the initial point and its dynamic behavior is more
unpredictable.

\begin{figure}
\centering
\includegraphics[width=0.3\textwidth]{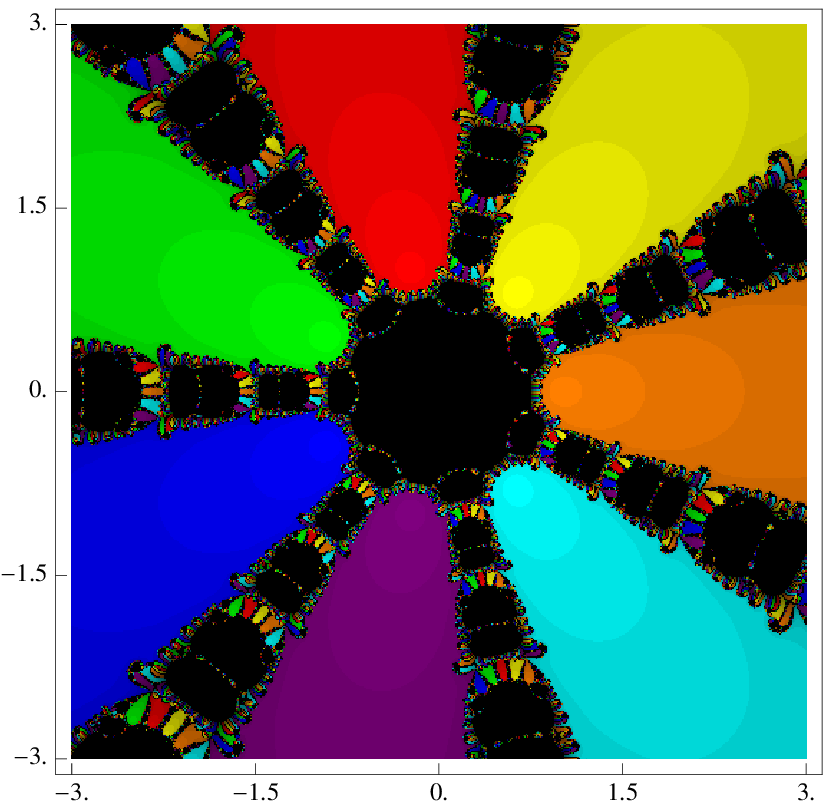}
\quad
\includegraphics[width=0.3\textwidth]{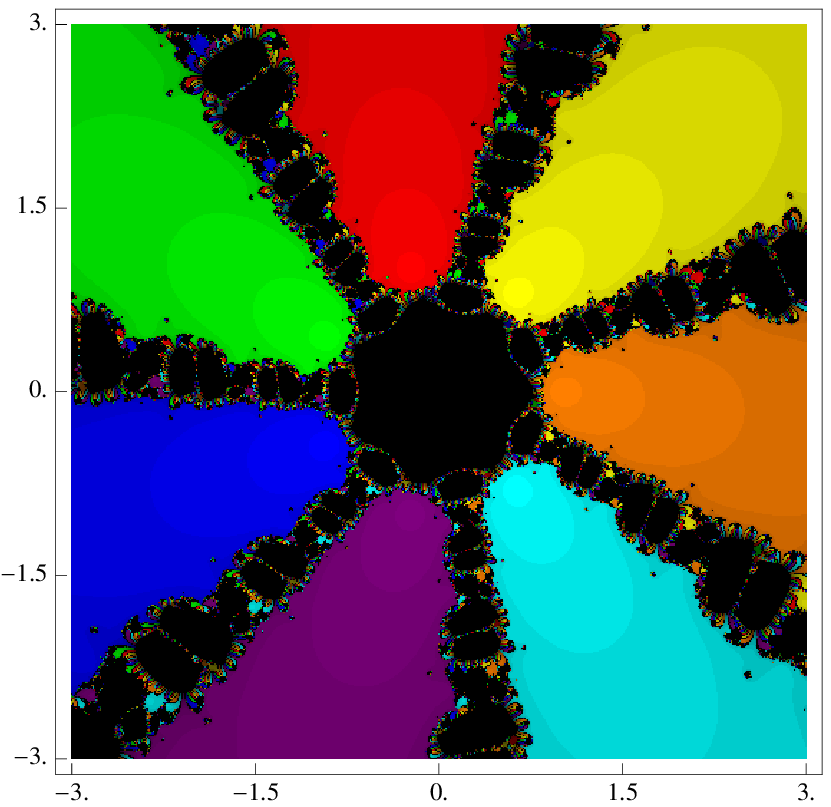}
\quad
\includegraphics[width=0.3\textwidth]{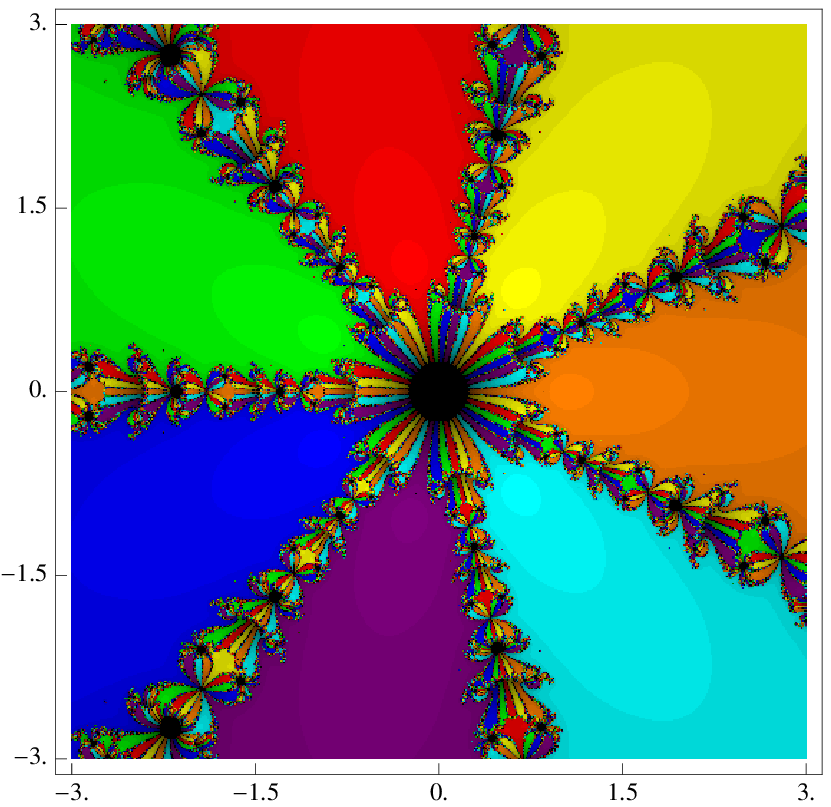}
\\
\includegraphics[width=0.3\textwidth]{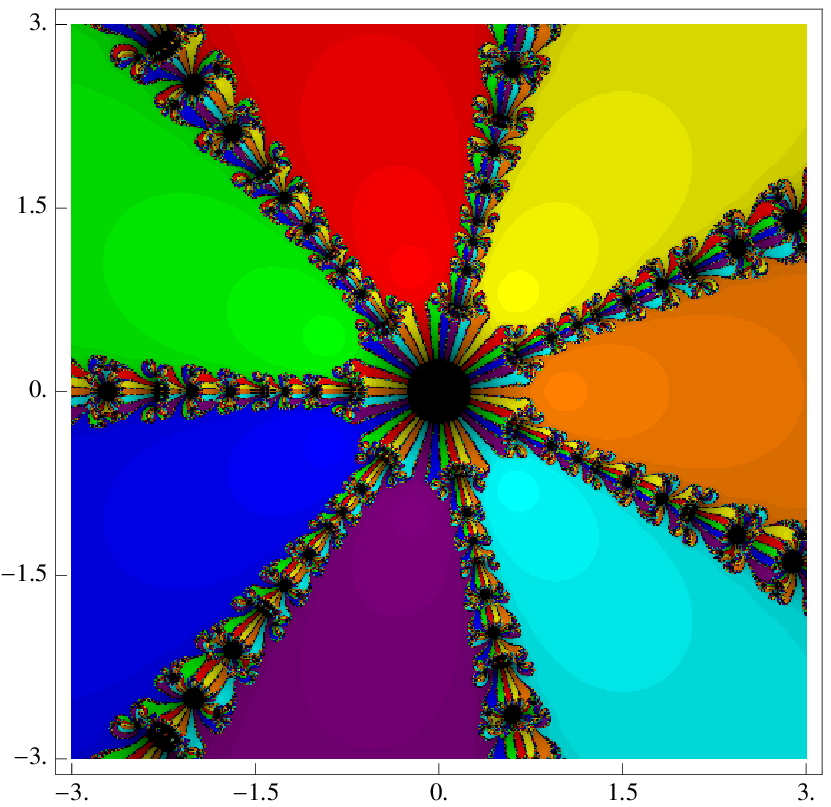}
\quad
\includegraphics[width=0.3\textwidth]{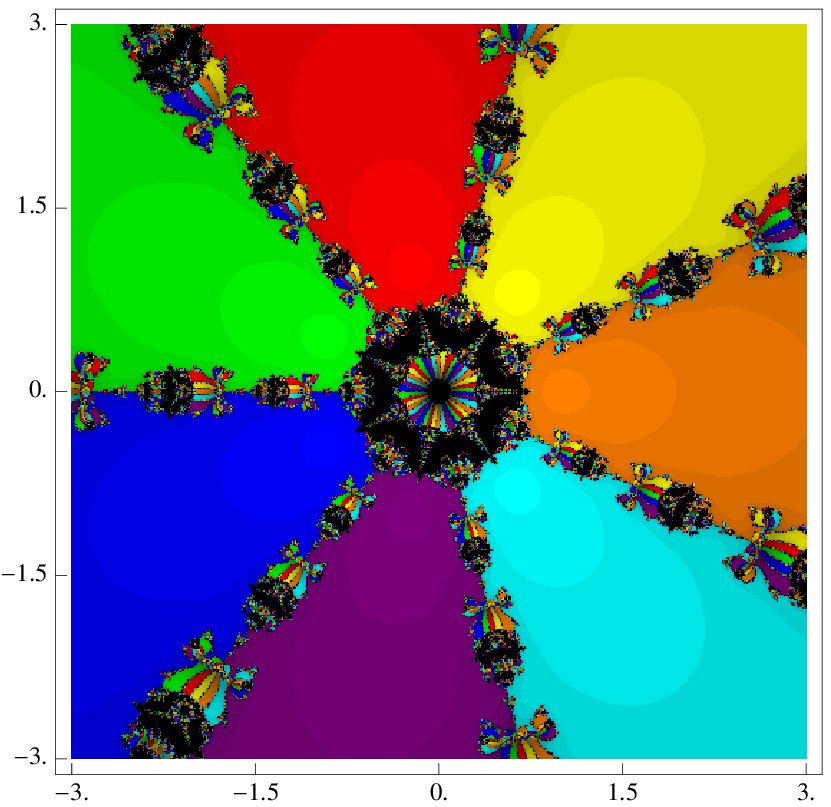}
\quad
\includegraphics[width=0.3\textwidth]{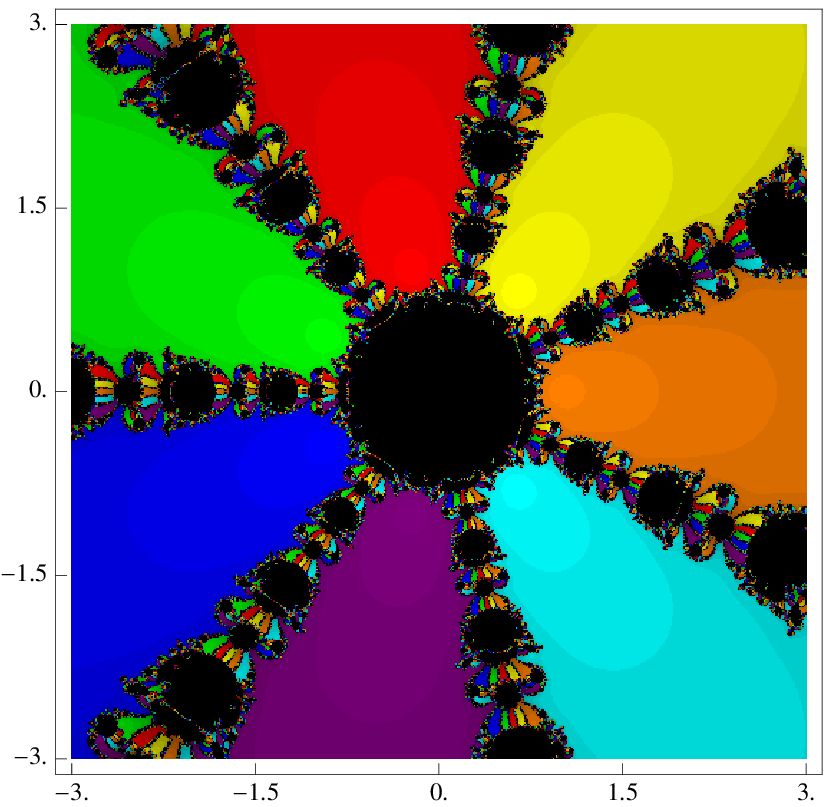}
\caption{Comparison of basins of attraction of methods
\eqref{n1}--\eqref{o4} for the test problem $p_5(z)= z^7-1=0$.}
\label{fig:figure5}
\end{figure}

\begin{figure}
\centering
\includegraphics[width=0.3\textwidth]{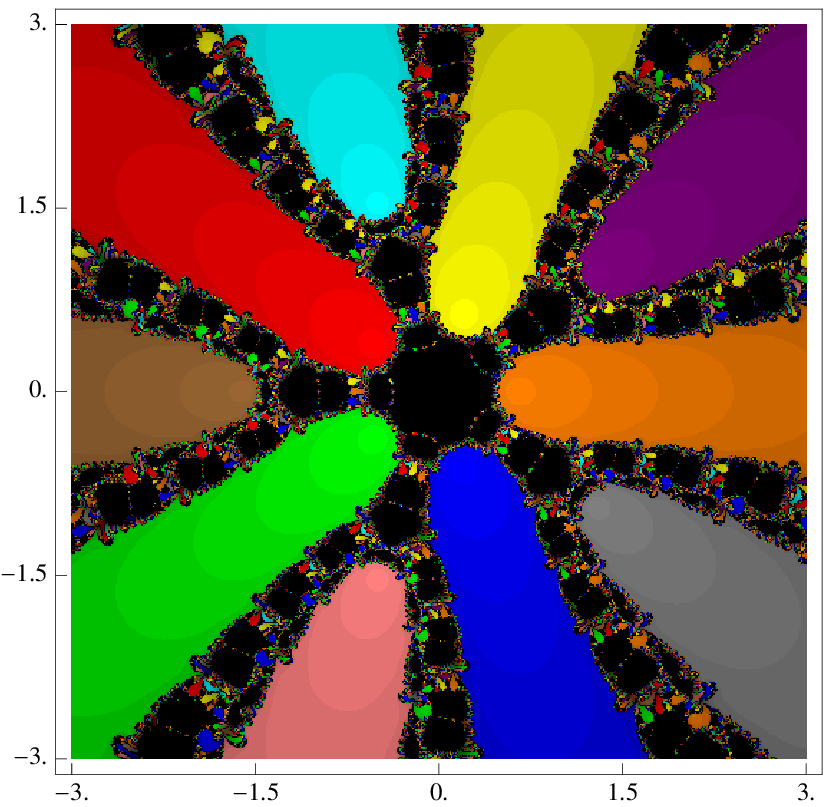}
\quad
\includegraphics[width=0.3\textwidth]{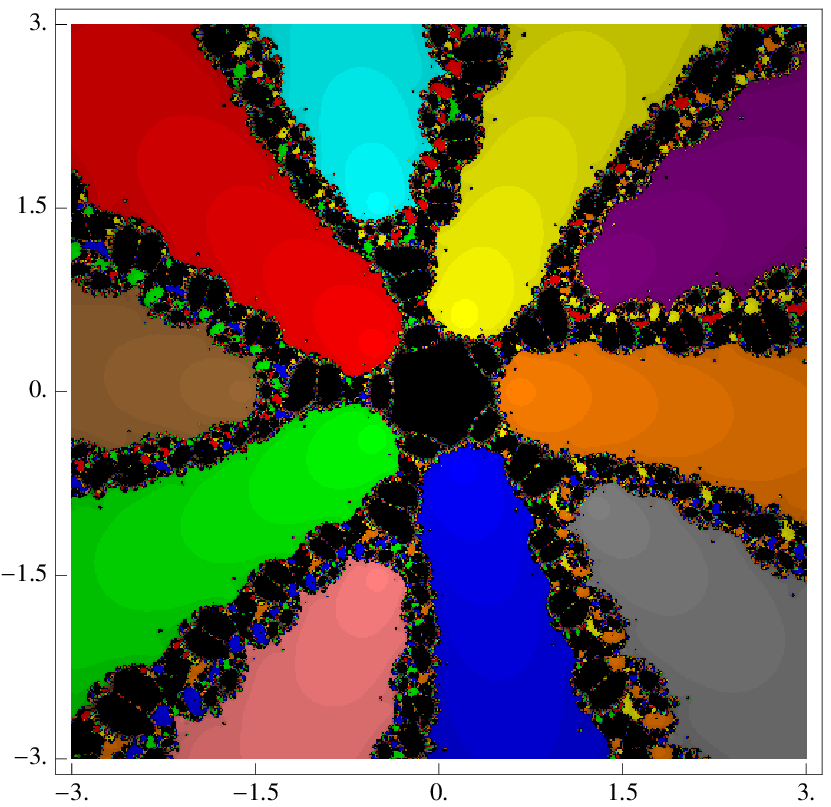}
\quad
\includegraphics[width=0.3\textwidth]{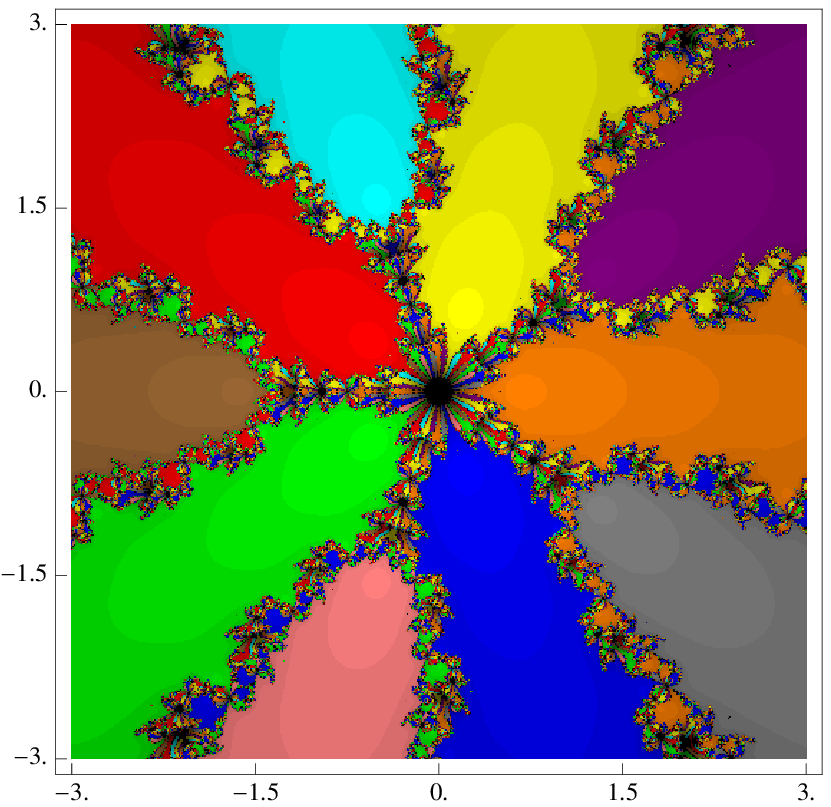}
\\
\includegraphics[width=0.3\textwidth]{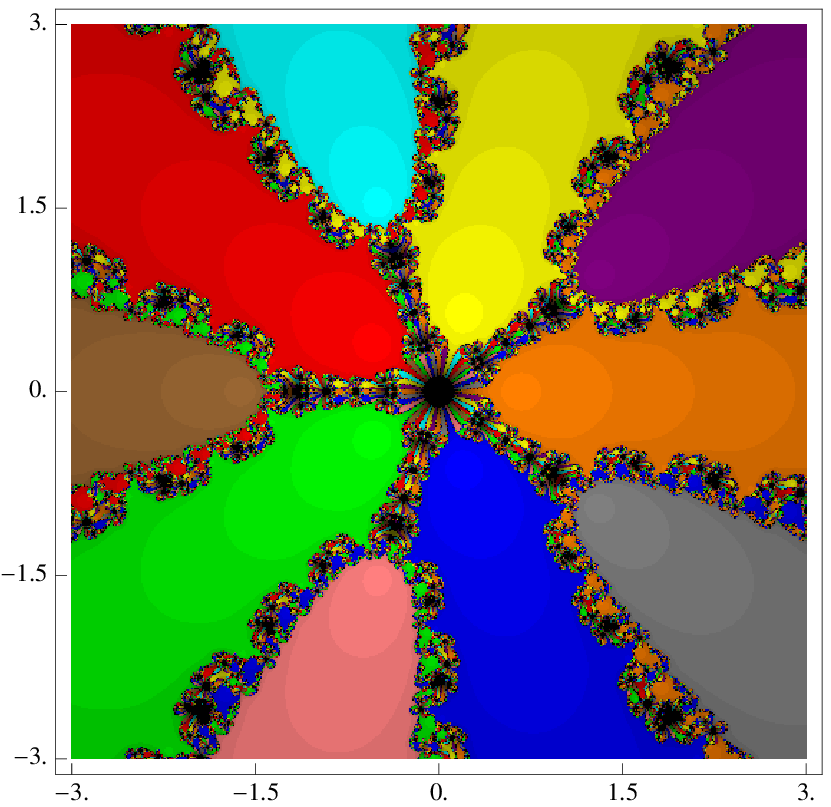}
\quad
\includegraphics[width=0.3\textwidth]{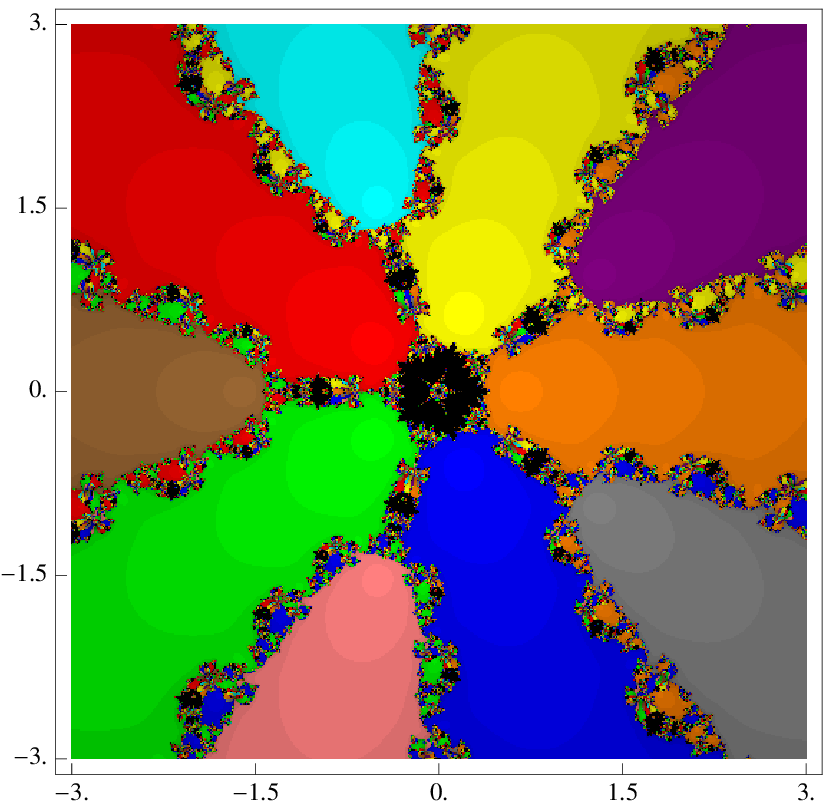}
\quad
\includegraphics[width=0.3\textwidth]{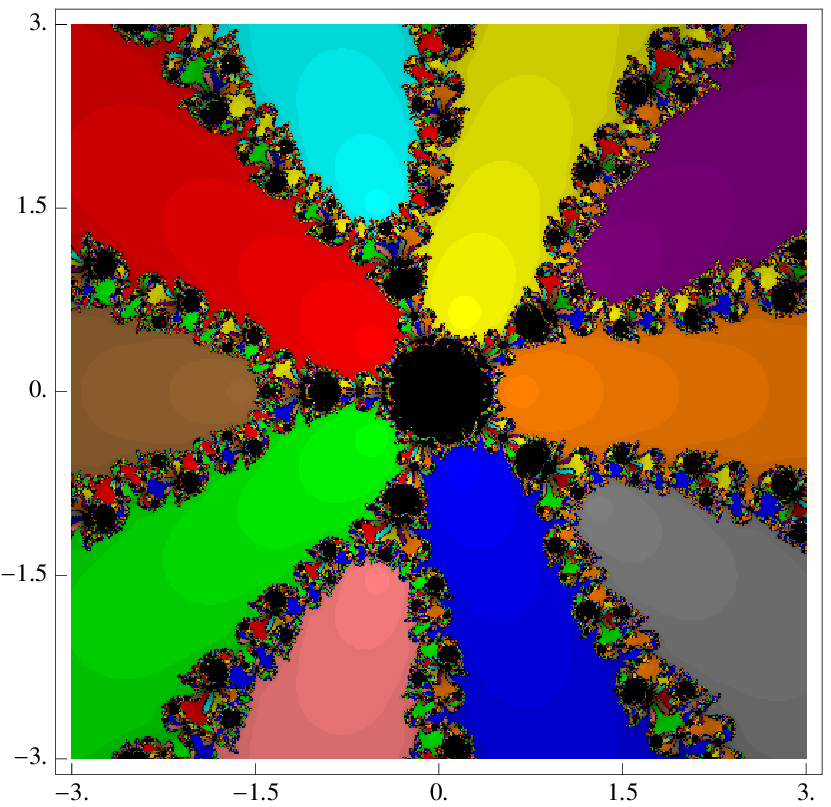}
\caption{Comparison of basins of attraction of methods
\eqref{n1}--\eqref{o4} for the test problem $p_6(z)=
(10z^5-1)(z^5+10)=0$.} 
\label{fig:figure6}
\end{figure}

The first and the second graphics in the
Figures~\ref{fig:figure1}--\ref{fig:figure6}
correspond to the same general method~\eqref{a3}
with different choices of the parameters $a,b,c$ in~\eqref{a36}
and~\eqref{a37}. We can see from the graphics in
Figure~\ref{fig:figure7} and the first two pictures of
Figure~\ref{fig:figure6} that even small changes in the parameters may lead
to completely different behaviors.
%Perhaps one can
%expect that, being small variation of the same weight functions
%$J$ and $G$, the dynamics should be very similar (closer in
%comparison with the methods~\eqref{o1}--\eqref{o4}), and this is
%not true in general. Actually, much more difference is possible
%with other choice of the parameters. To show it we have
%represented, in Figure~\ref{fig:figure7}, the basins of attraction
%of the test problem $p_6(z)=0$ with other choices of the
%parameters $a,b,c$. It is worthwhile to note that the aspect is
%completely different of what we observe in the two first graphics
%of Figure~\ref{fig:figure6}.

\begin{figure}
\centering
\includegraphics[width=0.3\textwidth]{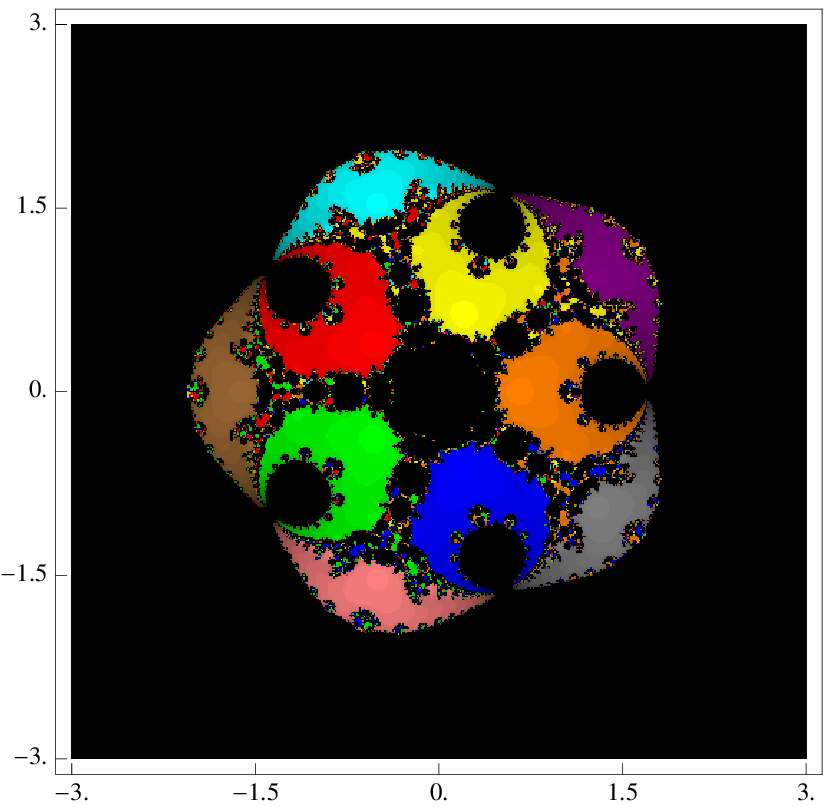}
\quad
\includegraphics[width=0.3\textwidth]{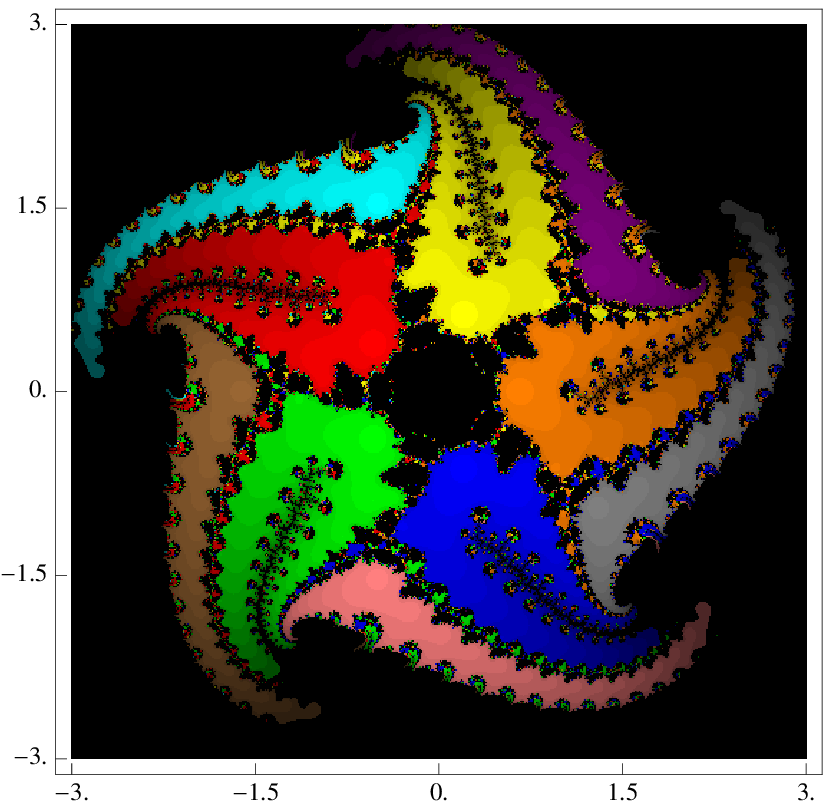}
\quad
\includegraphics[width=0.3\textwidth]{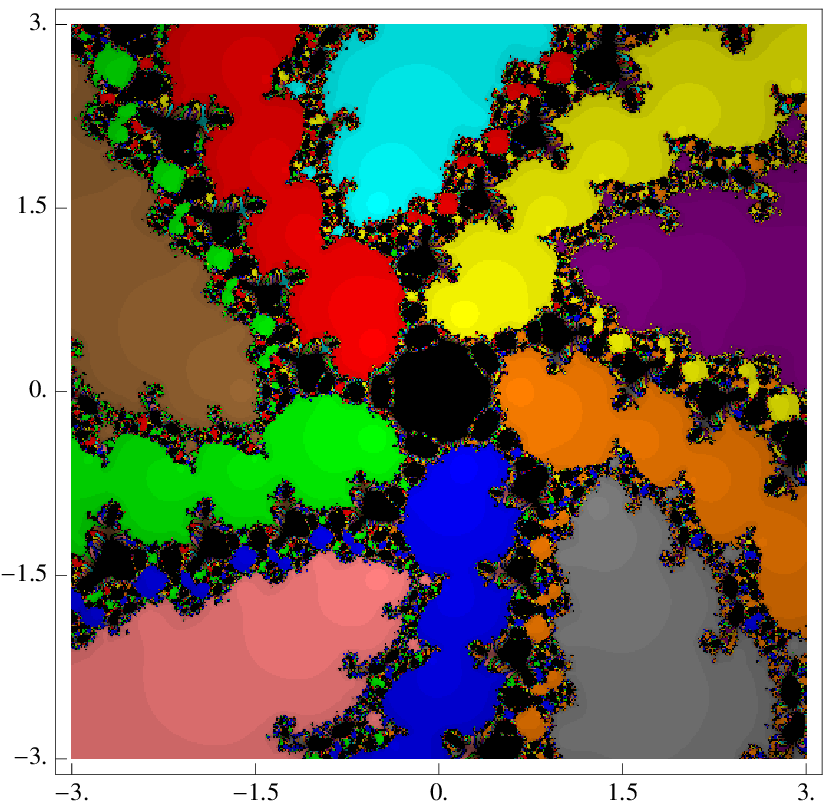}
\caption{Basins of attraction of the method of
Theorem~\ref{thm:main} to solve $p_6(z)= (10z^5-1)(z^5+10)=0$ with
different choices of the parameters $a$, $b$ and~$c$. Left:
$a=b=c=-1$. Center: $a=b=-1$, $c=-1+i$. Right: $a=-1/2$, $b=-1$,
$c=-2+i$.} 
\label{fig:figure7}
\end{figure}

Finally, we have included in Table~\ref{table5} the results of
some numerical experiments to measure the behavior of the six iterative
methods \eqref{n1}--\eqref{o4} in finding the roots of the test polynomials
$p_j(z)$, $j=1,\dots,6$. To compute the data of this table,
we have applied the six methods to the six polynomials, starting at an
initial points $z_0$ on a $512 \times 512$ grid in the rectangle
$[-3,3] \times [-3,3]$ of the complex plane. The same way was used in
Figures~\ref{fig:figure1}--\ref{fig:figure6} to show the basins
of attraction of the roots. In particular, we decide again that an
initial point $z_0$ has reached a root $z^*$ when its distance
to $z^*$ is less than $10^{-3}$ (in this case $z_0$ is in the
basin of attraction of~$z^*$) and we decide that the method
starting in $z_0$ diverges when no root is found in a maximum of
$15$ iterations of the method (in this case, we say that $z_0$ is
a ``nonconvergent point''). In Table~\ref{table5}, we have abbreviated the
methods \eqref{n1}--\eqref{o4} as M1--M6, respectively. The column
I/P shows the mean of iterations per point until the algorithm
decides that a root has been reached or the point is declared
nonconvergent. The column NC shows the percentage of nonconvergent
points, indicated as black zones in the figures. It is clear that the
nonconvergent points have a great influence on the values of I/P
since these points contribute always with the maximum number of $15$
allowed iterations, In contrast, ``convergent points'' are reached usually
very fast due to the fact that we are dealing with eighth-order methods.
To reduce the effect of nonconvergent points, we have included the column
I$_\textrm{C}$/C which shows the mean number of iterations per convergent
point. If we use either the columns I/P or the column I$_\textrm{C}$/C to
compare the performance of the iterative methods, we clearly obtain
different conclusions.

%\newpage

\begin{table}[htbp!]
\begin{center}
\begin{tabular}{cclll}
\toprule
Polynomial & Method & I/P & NC (\%) & I$_\textrm{C}$/C \\
\midrule
$p_1(z)$ & M1 & 2.53 & 0.244 & 2.50 \\
         & M2 & 2.29 & 0.00798 & 2.28 \\
         & M3 & 2.20 & 0.195 & 2.18 \\
         & M4 & 2.17 & 0.195 & 2.15 \\
         & M5 & 2.13 & 0.195 & 2.10 \\
         & M6 & 6.01 & 70.9 & 2.09 \\
\midrule
$p_2(z)$ & M1 & 3.54 & 0.798 & 3.45 \\
         & M2 & 3.10 & 0.340 & 3.06 \\
         & M3 & 2.88 & 0. & 2.88 \\
         & M4 & 2.82 & 0.00456 & 2.82 \\
         & M5 & 2.73 & 0. & 2.73 \\
         & M6 & 4.32 & 27.6 & 2.81 \\
\midrule
$p_3(z)$ & M1 & 3.88 & 3.57 & 3.47 \\
         & M2 & 3.57 & 2.19 & 3.31 \\
         & M3 & 2.99 & 0.0122 & 2.99 \\
         & M4 & 2.94 & 0.0334 & 2.94 \\
         & M5 & 2.82 & 0. & 2.82 \\
         & M6 & 3.28 & 5.46 & 2.99 \\
%\midrule
%$p_4(z)$ & M1 & 6.85 & 24.7 & 4.17 \\
%         & M2 & 6.48 & 22.0 & 4.07 \\
%         & M3 & 4.07 & 0.888 & 3.97 \\
%         & M4 & 4.21 & 1.84 & 4.01 \\
%         & M5 & 3.95 & 4.40 & 3.44 \\
%         & M6 & 4.45 & 20.1 & 3.56 \\
%\midrule
%$p_5(z)$ & M1 & 7.27 & 27.0 & 4.42 \\
%         & M2 & 7.00 & 25.2 & 4.30 \\
%         & M3 & 4.81 & 3.36 & 4.45 \\
%         & M4 & 5.07 & 5.71 & 4.47 \\
%         & M5 & 4.59 & 7.04 & 3.80 \\
%         & M6 & 5.03 & 21.4 & 4.02 \\
%\midrule
%$p_6(z)$ & M1 & 7.36 & 24.4 & 4.90 \\
%         & M2 & 6.96 & 21.7 & 4.73 \\
%         & M3 & 4.69 & 2.33 & 4.44 \\
%         & M4 & 4.89 & 4.03 & 4.46 \\
%         & M5 & 4.44 & 3.98 & 4.01 \\
%         & M6 & 5.26 & 11.9 & 4.70 \\
\bottomrule
\end{tabular}%
\quad
\begin{tabular}{cclll}
\toprule
Polynomial & Method & I/P & NC (\%) & I$_\textrm{C}$/C \\
%\midrule
%$p_1(z)$ & M1 & 2.53 & 0.244 & 2.50 \\
%         & M2 & 2.29 & 0.00798 & 2.28 \\
%         & M3 & 2.20 & 0.195 & 2.18 \\
%         & M4 & 2.17 & 0.195 & 2.15 \\
%         & M5 & 2.13 & 0.195 & 2.10 \\
%         & M6 & 6.01 & 70.9 & 2.09 \\
%\midrule
%$p_2(z)$ & M1 & 3.54 & 0.798 & 3.45 \\
%         & M2 & 3.10 & 0.340 & 3.06 \\
%         & M3 & 2.88 & 0. & 2.88 \\
%         & M4 & 2.82 & 0.00456 & 2.82 \\
%         & M5 & 2.73 & 0. & 2.73 \\
%         & M6 & 4.32 & 27.6 & 2.81 \\
%\midrule
%$p_3(z)$ & M1 & 3.88 & 3.57 & 3.47 \\
%         & M2 & 3.57 & 2.19 & 3.31 \\
%         & M3 & 2.99 & 0.0122 & 2.99 \\
%         & M4 & 2.94 & 0.0334 & 2.94 \\
%         & M5 & 2.82 & 0. & 2.82 \\
%         & M6 & 3.28 & 5.46 & 2.99 \\
\midrule
$p_4(z)$ & M1 & 6.85 & 24.7 & 4.17 \\
         & M2 & 6.48 & 22.0 & 4.07 \\
         & M3 & 4.07 & 0.888 & 3.97 \\
         & M4 & 4.21 & 1.84 & 4.01 \\
         & M5 & 3.95 & 4.40 & 3.44 \\
         & M6 & 4.45 & 20.1 & 3.56 \\
\midrule
$p_5(z)$ & M1 & 7.27 & 27.0 & 4.42 \\
         & M2 & 7.00 & 25.2 & 4.30 \\
         & M3 & 4.81 & 3.36 & 4.45 \\
         & M4 & 5.07 & 5.71 & 4.47 \\
         & M5 & 4.59 & 7.04 & 3.80 \\
         & M6 & 5.03 & 21.4 & 4.02 \\
\midrule
$p_6(z)$ & M1 & 7.36 & 24.4 & 4.90 \\
         & M2 & 6.96 & 21.7 & 4.73 \\
         & M3 & 4.69 & 2.33 & 4.44 \\
         & M4 & 4.89 & 4.03 & 4.46 \\
         & M5 & 4.44 & 3.98 & 4.01 \\
         & M6 & 5.26 & 11.9 & 4.70 \\
\bottomrule
\end{tabular}
\caption{Measures of convergence of the iterative methods
\eqref{n1}--\eqref{o4} (abbreviated as M1--M6) applied to find the
roots of the polynomials $p_j(z)$, $j=1,\dots,6$.}
\label{table5}
\end{center}
\end{table}

%--------------------------------------------------------------%
\section{Conclusion}
\label{sec:conclusion}
%--------------------------------------------------------------%

We have introduced a new optimal class of three-point methods without
memory for approximating a simple root of a given nonlinear
equation which use only four function evaluations each iteration and result
in a method of convergence order eight. Therefore, the Kung and
Traub's conjecture is supported. Numerical examples and
comparisons with some existing eighth-order methods are included and
confirm the theoretical results. The numerical experience
suggests that the new class is a valuable alternative for
solving these problems and finding simple roots. We used the basins
of attraction for comparing the iterative algorithms and we have included
some tables with comparative results.

%--------------------------------------------------------------%


\begin{thebibliography}{99}
%--------------------------------------------------------------%

\small

%\bibitem{Amat1}
%Amat, S., Busquier, S., Plaza, S.:
%Iterative root-finding methods,
%Unpublished report (2004).

%% The final publication of this "Unpublished report" is the following:

% MR2127479
%\bibitem{Amat1}
%Amat, S., Busquier, S., Plaza, S.: Review of some iterative
%root-finding methods from a dynamical point of view, Sci. Ser. A
%Math. Sci. (N.S.) 10, 3--35 (2004).
%% SCIENTIA, Series A: Math. Sciences \textbf{10}, 3--35 (2004).
% http://www.mat.utfsm.cl/scientia/
% http://www.mat.utfsm.cl/scientia/archivos/vol10/part_2.pdf

% MR2139283
%\bibitem{Amat4} Amat, S., Busquier, S., Plaza, S.: Dynamics of
%the King and Jarratt iterations, Aequationes Math. 69, 212--223
%(2005).

\bibitem{Amat5}
Amat, S., Busquier, S., Magre\~n\'an, \'A.A.: Reducing chaos and
bifurcations in Newton-type methods, Abstr. Appl. Anal. 2013,
Art.\ ID 726701, 10 pages (2013).

\bibitem{Babajee}
Babajee, D.K.R., Cordero, A., Soleymani, F., Torregrosa, J.R.: On
improved three-step schemes with high efficiency index and their
dynamics, Numer. Algorithms 65, 153--169 (2014).

% MR2490175
\bibitem{Bi}
Bi, W., Ren, H., Wu, Q.: Three-step iterative methods with
eighth-order convergence for solving nonlinear equations, J.
Comput. Appl. Math. 225, 105--112 (2009).

\bibitem{Chicharro}
Chicharro, F., Cordero, A., Guti\'errez, J.M., Torregrosa, J.R.:
Complex dynamics of derivative-free methods for nonlinear
equations, Appl. Math. Comput. 219, 7023--7035 (2013).

\bibitem{Chun1}
Chun, C., Lee, M.Y.: A new optimal eighth-order family of
iterative methods for the solution of nonlinear equations, Appl.
Math. Comput. 223, 506--519 (2013).

\bibitem{acoc}
Cordero, A., Torregrosa, J.R.: Variants of Newton''s method using
fifth-order quadrature formulas, Appl. Math. Comput. 190, 686--698
(2007).

\bibitem{Ezquerro2}
Ezquerro, J.A., Hern\'andez, M.A.: An improvement of the region of
accessibility of Chebyshev's method from Newton's method, Math.
Comp. 78, 1613--1627 (2009).

\bibitem{Ezquerro1}
Ezquerro, J.A., Hern\'andez, M.A.: An optimization of Chebyshev's
method, J. Complexity 25, 343--361 (2009).

\bibitem{acoc-et-al}
Grau-S\'anchez, M., Noguera, M., Guti\'errez, J.M.: On some
computational orders of convergence, Appl. Math. Lett. 23(4),
472--478 (2010).

\bibitem{GutiMaVar}
Guti\'errez, J.M., Magre\~n\'an, \'A.A., Varona, J.L.: The
``Gauss-Seidelization'' of iterative methods for solving nonlinear
equations in the complex plane, Appl. Math. Comput. 218,
2467--2479 (2011).

%\bibitem{Hazrat}
%Hazrat, R.:
%Mathematica\textsuperscript{\textregistered}: A Problem-Centered Approach,
%Springer-Verlag (2010).

\bibitem{Paricio}
Hern\'andez-Paricio, L.J., Mara\~n\'on-Grandes, M.,
Rivas-Rodr\'iguez, M.T.: Plotting basins of end points of rational
maps with Sage, Tbil. Math. J. 5(2), 71--99 (2012).
% http://tcms.org.ge/Journals/TMJ/

% http://www.ams.org/journals/mcom/1966-20-095/S0025-5718-66-99924-8/
\bibitem{Jarratt}
Jarratt, P.: Some fourth order multipoint iterative methods for
solving equations, Math. Comp. 20, 434--437 (1966).

% MR0343585
\bibitem{King}
King, R.F.: A family of fourth order methods for nonlinear
equations, SIAM J. Numer. Anal. 10, 876--879 (1973).

% MR0353657
\bibitem{Kung}
Kung, H.T., Traub, J.F.: Optimal order of one-point and multipoint
iteration, J. Assoc. Comput. Mach. 21, 634--651 (1974).

% MR3304844
\bibitem{Lotfi2}
Lotfi, T., Sharifi, S., Salimi, M., Siegmund, S.: A new class of
three-point methods with optimal convergence order eight and its
dynamics, Numer. Algorithms 68, 261--288 (2015).

\bibitem{Neta0}
Neta, B.: On a family of multipoint methods for nonlinear
equations, Internat. J. Comput. Math. 9, 353--361 (1981).

% MR3146342
\bibitem{Neta1}
Neta, B., Chun, C., Scott, M.: Basins of attraction for optimal
eighth order methods to find simple roots of nonlinear equations,
Appl. Math. Comput. 227, 567--592 (2014).

% MR0216746
\bibitem{Ostrowski}
Ostrowski, A.M.:
Solution of Equations and Systems of Equations, 2nd ed.,
Academic Press, New York (1966).

\bibitem{Petkovic}
Petkovi\'c, M.S., Neta, B., Petkovi\'c, L.D., D\v{z}uni\'c, J.:
Multipoint Methods for Solving Nonlinear Equations,
Elsevier/Academic Press, Amsterdam (2013).

% MR2838167
%\bibitem{Scott}
%Scott, M., Neta, B., Chun, C.: Basin attractors for various
%methods, Appl. Math. Comput. 218, 2584--2599 (2011).

\bibitem{Sharifi1}
Sharifi, S., Ferrara, M., Salimi, M., Siegmund, S.: New
modification of Maheshwari method with optimal eighth order of
convergence for solving nonlinear equations, preprint (2015).
% Under revision (2015).

\bibitem{Sharifi2}
Sharifi, S., Siegmund, S., Salimi, M.: Solving nonlinear equations
by a derivative-free form of the King's family with memory,
Calcolo, doi: 10.1007/s10092-015-0144-1 (2015).

\bibitem{Sharma1}
Sharma, J.R., Sharma, R.: A new family of modified Ostrowski's
methods with accelerated eighth order convergence, Numer.
Algorithms 54, 445--458 (2010).

\bibitem{Stewart}
Stewart, B.D.:
Attractor Basins of Various Root-Finding Methods,
M.S. thesis, Naval Postgraduate School,
% Department of Applied Mathematics,
Monterey, CA (2001).
% http://calhoun.nps.edu/bitstream/handle/10945/10997/ADA397512.pdf
% http://hdl.handle.net/10945/10997

% MR0169356
\bibitem{Traub}
Traub, J.F.:
Iterative Methods for the Solution of Equations,
Prentice Hall, Englewood Cliffs, N.J. (1964).

\bibitem{Varona}
Varona, J.L.: Graphic and numerical comparison between iterative
methods, Math. Intelligencer 24(1), 37--46 (2002).

% MR0918313
%\bibitem{Vrscay}
%Vrscay, E.R., Gilbert, W.J.: Extraneous fixed points, basin
%boundaries and chaotic dynamics for Schr\"oder and K\"onig
%rational iteration functions, Numer. Math. 52, 1--16 (1988).

% MR2610375
\bibitem{Wang}
Wang, X., Liu, L.: New eighth-order iterative methods for solving
nonlinear equations, J. Comput. Appl. Math. 234, 1611--1620
(2010).

\bibitem{coc}
Weerakoon, S., Fernando, T.G.I.: A variant of Newton's method with
accelerated third-order convergence, Appl. Math. Lett. 13(8),
87--93 (2000).

\end{thebibliography}
\end{document}